\documentclass[twoside,11pt,reqno]{amsart}
\usepackage{amsmath,amssymb,amscd,mathrsfs,amscd, todonotes,appendix}
\usepackage{graphics,verbatim}
\usepackage{todonotes}
\usepackage{enumitem}
\usepackage{hyperref}

\makeatletter
\@namedef{subjclassname@2020}{\textup{2020} Mathematics Subject Classification}
\makeatother

\newcommand{\losemi}{{\otimes \kern -.78em \ltimes}}
\newcommand{\rosemi}{{\otimes \kern -.78em \rtimes}}

\newcommand{\Hom}{\ensuremath{\operatorname{Hom}}}

\newcommand{\Ext}{\operatorname{Ext}}

\newcommand{\ind}{\operatorname{ind}}
\newcommand{\rad}{\operatorname{rad}}

\renewcommand{\a}{\alpha}

\newcommand{\ga}{\gamma}
\newcommand{\s}{\sigma}

\newcommand{\la}{\lambda}

\newcommand{\St}{\operatorname{St}}

\newcommand{\qq}{\widehat{Q}_{1}}
\newcommand{\soc}{\operatorname{soc}}

\newcommand{\pp}{^{\oplus 2}}
\newcommand{\ppp}{^{\oplus 3}}

\makeatletter
\newcommand{\leqnomode}{\tagsleft@true}
\newcommand{\reqnomode}{\tagsleft@false}
\makeatother

\newtheorem{theorem}{Theorem}[subsection]

\makeatletter\let\c@fact\c@theorem\makeatother

\makeatletter\let\c@note\c@theorem\makeatother

\newtheorem{lemma}{Lemma}[subsection]
\makeatletter\let\c@lemma\c@theorem\makeatother

\makeatletter\let\c@lemma\c@theorem\makeatother

\makeatletter\let\c@alg\c@theorem\makeatother

\newtheorem{prop}{Proposition}[subsection]
\makeatletter\let\c@prop\c@theorem\makeatother

\newtheorem{conj}{Conjecture}[subsection]
\makeatletter\let\c@conj\c@theorem\makeatother

\newtheorem{cor}{Corollary}[subsection]
\makeatletter\let\c@cor\c@theorem\makeatother

\makeatletter\let\c@defn\c@theorem\makeatother

\theoremstyle{definition}

\newtheorem{remark}{Remark}[subsection]
\makeatletter\let\c@remark\c@theorem\makeatother

\makeatletter\let\c@example\c@theorem\makeatother
\numberwithin{equation}{subsection}

\usepackage[capitalise]{cleveref}
\crefname{theorem}{Theorem}{Theorems}
\crefname{fact}{Fact}{Facts}
\crefname{note}{Note}{Notes}
\crefname{lemma}{Lemma}{Lemmas}
\crefname{alg}{Algorithm}{Algorithms}
\crefname{remark}{Remark}{Remarks}
\crefname{example}{Example}{Examples}
\crefname{prop}{Proposition}{Propositions}
\crefname{conj}{Conjecture}{Conjectures}
\crefname{cor}{Corollary}{Corollaries}
\crefname{defn}{Definition}{Definitions}
\crefname{equation}{\!\!}{\!\!} 

\newcounter{listequation}


\begin{document}

\title{On Donkin's Tilting Module Conjecture I: Lowering the prime }

\dedicatory{In memory of James E. Humphreys}

\begin{abstract} In this paper the authors provide a complete answer to  Donkin's Tilting Module Conjecture for all rank $2$ semisimple algebraic 
groups and $\text{SL}_{4}(k)$ where $k$ is an algebraically closed field of characteristic $p>0$. In the process, new techniques are introduced 
involving the existence of $(p,r)$-filtrations, Lusztig's character formula, and the $G_{r}$T-radical series for baby Verma modules. 
\end{abstract}

\author{\sc Christopher P. Bendel}
\address
{Department of Mathematics, Statistics and Computer Science\\
University of
Wisconsin-Stout \\
Menomonie\\ WI~54751, USA}
\thanks{Research of the first author was supported in part by Simons Foundation Collaboration Grant 317062}
\email{bendelc@uwstout.edu}

\author{\sc Daniel K. Nakano}
\address
{Department of Mathematics\\ University of Georgia \\
Athens\\ GA~30602, USA}
\thanks{Research of the second author was supported in part by
NSF grants DMS-1701768 and DMS-2101941}
\email{nakano@math.uga.edu}

\author{\sc Cornelius Pillen}
\address{Department of Mathematics and Statistics \\ University
of South
Alabama\\
Mobile\\ AL~36688, USA}
\thanks{Research of the third author was supported in part by Simons Foundation Collaboration Grant 245236}
\email{pillen@southalabama.edu}

\author{Paul Sobaje}
\address{Department of Mathematical Sciences \\
          Georgia Southern University\\
          Statesboro, GA~30458, USA}
\email{psobaje@georgiasouthern.edu}

\maketitle

\section{Introduction}
		
\subsection{} In 1990, Donkin at an MSRI conference stated a series of conjectures about reductive algebraic group representations in characteristic $p>0$. One of the 
conjectures known as the Tilting Module Conjecture (TMC) (see Conjecture~\ref{C:tilting}) states that a projective indecomposable module for the Frobenius kernel $G_{r}$ of a semisimple  group $G$ can be realized as the restriction of a specific tilting module for $G$. A solution to this conjecture implies a positive answer to the Humphreys-Verma Conjecture  (cf. 
\cite[10.4 Question]{Hum}) that such a projective indecomposable $G_r$-module admits a $G$-structure. 

There have been numerous attempts over the past 30 years to prove the Humphreys-Verma Conjecture and later the Tilting Module Conjecture. Ballard \cite{B} first proved the Humphreys-Verma Conjecture for $p\geq 3h-3$ and Jantzen \cite{J} lowered this bound to $p\geq 2h-2$, where $h$ is the Coxeter number associated to the root system $\Phi$ for $G$. Jantzen's general lower bound has long stood as the sharpest known lower bound, even though it was suspected that the Tilting Module Conjecture should hold for all primes. In an unexpected breakthrough in 2019, the authors of this paper discovered the first counterexample to the TMC when the group has root system $\Phi$ of type $\rm{G}_{2}$ and $p=2$. 

Many people have tried unsuccessfully to lower the bound on $p$ and the only other cases that the TMC was known to hold for all primes are when $\Phi=\rm{A}_{1}$ or $\rm{A}_{2}$. 
The goal of this paper is to introduce  new ideas and techniques that enable us to verify the TMC in many other cases. These ideas originate from the fundamental work of 
Kildetoft-Nakano \cite{KN} and Sobaje \cite{So} that relates the TMC to the existence of good $(p,r)$-filtrations. As a byproduct of our work, we provide a new proof of Jantzen's lower bound. Furthermore, we provide a complete answer for all rank $2$ groups and $\Phi=\rm{A}_{3}$ for all primes.

\begin{theorem}\label{T:MainTheorem} Let $G$ be a simple, simply connected algebraic group scheme defined and split over ${\mathbb F}_{p}$ and $\Phi$ be its associated root system. The Tilting Module Conjecture 
holds if 
\begin{itemize} 
\item[(a)] $\Phi=\rm{A}_{n}$, $n\leq 3$;
\item[(b)] $\Phi=\rm{B}_{2}$;
\item[(c)] $\Phi=\rm{G}_{2}$ as long as $p\neq 2$.
\end{itemize} 
\end{theorem} 

The most difficult case to verify in Theorem~\ref{T:MainTheorem} is when $\Phi=\rm{G}_{2}$, $p=7$. The verification of the TMC in this case employs several deep results which include the validity of Lusztig's Conjecture and the description of the $G_{r}T$-radical series of baby Verma modules 
(provided in the Appendix) via the computation of inverse Kazhdan-Lusztig polynomials. 

The authors have recently constructed examples for $\Phi=\rm{B}_{3}$ ($p=2$) and $\Phi=\rm{C}_{3}$ ($p=3$) where the TMC fails, so the case when $G$ has rank $3$ is much more complicated. These counterexamples along with others will appear in a forthcoming paper \cite{BNPS3}. With new insights, techniques and examples, we aim to be able to determine all primes for which the TMC holds.

The work in this paper also has application to character formulas for $G$-modules.  In \cite{RW2}, Riche and Williamson proved for all $p$ that the characters of indecomposable tilting modules can be given via $p$-Kazhdan-Lusztig polynomials, extending the results from \cite{RW} and \cite{AMRW}.  In \cite{So2}, it was shown that one can always choose an $r \ge 1$ large enough (depending on $G$) such that the characters of the simple $G$-modules can be derived from the characters of the tilting modules with highest weights in the set $(p^r-1)\rho+X_1$, where $\rho$ denotes the Weyl weight and $X_1$ denotes the set of $p$-restricted dominant weights.  The bound on $r$ is in general not optimal, which is a drawback from a computational perspective given that the complexity in finding tilting characters grows substantially as the highest weight increases. The TMC is a key statement because 
the optimal bound of $r=1$ is achieved precisely when Donkin's Tilting Module Conjecture is valid. 

\subsection{} The paper is organized as follows. In Section~\ref{S:filt-trans}, we provide the notation that will be used throughout the paper and formally state the Tilting Module Conjecture. We then prove a result about the existence of good $(p,r)$-filtrations for induced modules and how methods involving translation functors reduced our consideration of the TMC to regular weights (when $p \ge h$). 

In the following section (Section \ref{S:weightThms}), we start by presenting a general theorem involving weight combinatorics and root pairings that can be employed to verify the TMC. Later, in this section, it is shown how this  theorem can be used to recover Jantzen's lower bound by proving the TMC for $p\geq 2h-2$. In general, the conditions in the theorem break down for smaller primes, but due to some fortunate circumstances, this approach can be later used to verify the TMC for $\Phi = \rm{B}_2$ with $p = 3$ and $\Phi=\rm{G}_2$ with $p=5$. At the end of the section, the case when $p=2h-3$ is analyzed and results are provided for when the TMC holds. 

Section 4 is devoted to investigating the intriguing connection between splitting of maps, the existence of good $(p,r)$-filtrations, and the validity of the TMC. This analysis culminates in Theorem \ref{T:summary} which gives sufficient representation-theoretic conditions on a finite set of modules for the TMC to hold. The main theorem of the paper (Theorem~\ref{T:MainTheorem}) is proved in 
Sections 5 through 8. It is important to note that one single argument cannot handle all these cases and we appeal to a myriad of old and new techniques to treat each case separately. 

\subsection{Acknowledgements} We thank David Stewart and the referees for comments and suggestions on an earlier version of our manuscript.


\section{Filtrations and Translation}\label{S:filt-trans}

\subsection{\bf Notation.}\label{S:Notation} The notation for the most part follows the standard conventions in \cite{rags}.\footnote{The notation for the induced and Weyl modules 
is provided via the costandard and standard module conventions in the highest weight category literature.} Let $G$ be a connected semisimple algebraic group scheme 
defined and split over ${\mathbb F}_{p}$ with Frobenius morphism $F$. The $r$th Frobenius kernel will be denoted by $G_{r}$, and its graded version by $G_{r}T$.  Given a split maximal torus $T$, let $X$ be the set of weights 
for $G$, $X_{+}$ be the dominant weights for $G$, and $X_{r}$ be the $p^{r}$-restricted weights. For $\lambda\in X_{+}$, there are four fundamental families of $G$-modules (each having highest weight $\lambda$): 
$L(\lambda)$ (simple), $\nabla(\lambda)$ (costandard/induced), $\Delta(\lambda)$ (standard/Weyl), and $T(\lambda)$ (indecomposable tilting).

Let $\tau:G \rightarrow G$ be the Chevalley antiautomorphism of $G$ that is the identity morphism when restricted to $T$ (see \cite[II.1.16]{rags}).  Given a finite dimensional $G$-module $M$ (over a field $k$ of characteristic $p$), the module $^{\tau}M$ is $M^*$ (the ordinary $k$-linear dual of $M$) as a $k$-vector space, with action $g.f(m)=f(\tau(g).m)$.  This defines a functor from $G$-mod to $G$-mod that preserves the character of $M$.  In particular, it is the identity functor on all simple and tilting modules.

For each $\mu \in X$ and positive integer $r$, there is a {\em baby Verma module}:
$$\widehat{Z}_r^{\prime}(\mu) := \text{ind}_B^{G_rB} \mu.$$

Let $\rho$ be the sum of the fundamental weights and $\text{St}_r = L((p^r-1)\rho)$ be the $r$th Steinberg module. For $\lambda\in X_{r}$, let $Q_{r}(\lambda)$ denote the $G_{r}$-projective cover (equivalently, injective hull) of $L(\lambda)$ as a $G_{r}$-module. For $\lambda\in X$, if $\widehat{L}_{r}(\lambda)$ is the corresponding simple $G_{r}T$-module, let $\widehat{Q}_{r}(\lambda)$ denote the $G_{r}T$-projective cover (equivalently, injective hull) of $\widehat{L}_{r}(\lambda)$. 
If $\lambda\in X_{r}$, set 
$$\hat{\lambda}=2(p^{r}-1)\rho+w_{0}\lambda$$ 
where $w_{0}$ is the longest element in the Weyl group $W$. Let $h$ denote the Coxeter number for the root system associated to $G$. 

We need to introduce another important class of modules. For $\lambda\in X_+$ with unique decomposition $\lambda = \lambda_0 + p^r\lambda_1$ with $\lambda_0\in X_r$ 
and $\lambda_1\in X_+$, define $\nabla^{(p,r)}(\lambda) = L(\lambda_0)\otimes \nabla(\lambda_1)^{(r)}$ where $(r)$ denotes the twisting of the module action by the 
$r$th Frobenius morphism. A $G$-module $M$ has a {\em good filtration} 
(resp. {\em good $(p,r)$-filtration}) if and only if $M$ has a filtration with factors of the form $\nabla(\mu)$ (resp. $\nabla^{(p,r)}(\mu)$) for suitable $\mu\in X_+$.  
In the case when $r=1$, good $(p,1)$-filtrations are often referred to as good $p$-filtrations. The similar notion of a Weyl $(p,r)$-filtration can be defined using the modules 
$\Delta^{(p,r)}(\lambda) = L(\lambda_0)\otimes \Delta(\lambda_1)^{(r)}$\footnote{This notation is used in \cite{KN} to denote a different class of modules.}. 

\subsection{Donkin's Tilting Module Conjecture} Given the notation in the preceding section, we can formally state 
Donkin's Tilting Module Conjecture. 
 
\begin{conj}\label{C:tilting} For all $\lambda\in X_{r}$, $T(\hat{\lambda})|_{G_{r}T}\cong \widehat{Q}_{r}(\lambda)$. 
\end{conj} 
Alternatively, a more symmetric statement of the conjecture is that 
$$T((p^{r}-1)\rho+\lambda)|_{G_{r}T}\cong\widehat{Q}_{r}((p^{r}-1)\rho+w_{0}\lambda)$$
for all $\lambda\in X_{r}$. We remark that the statement of Conjecture~\ref{C:tilting} can also be formulated by replacing $G_{r}T$ by $G_{r}$, and $\widehat{Q}_{r}(\lambda)$ with 
$Q_{r}(\lambda)$. The tilting module $T(\hat{\lambda})$ has $L(\lambda)$ appearing with multiplicity one in both its $G$-socle and $G_r$-socle (and by duality, appears with multiplicity one in 
its semisimple quotients over $G$ and $G_r$ respectively).

We now demonstrate that to verify the TMC holds for $r\geq 1$, it suffices to show that it holds for $r=1$.  

\begin{prop}\label{P:equivTMCr=1}
The TMC holds for $G_1$ if and only if it holds for $G_r$ for all $r > 1$.
\end{prop}

\begin{proof} Suppose that the TMC holds for $G_1$.
Let $\lambda \in X_r$, and write $\lambda = \lambda_0 + p\lambda_1 + \cdots + p^{r-1}\lambda_{r-1}$ with each $\lambda_i \in X_1$.  From \cite[II.11.15]{rags}, one obtains by induction that there is an isomorphism of $G_1T$-modules
$$\widehat{Q}_r(\lambda) \cong \widehat{Q}_1(\lambda_0) \otimes \widehat{Q}_1(\lambda_1)^{(1)} \otimes \cdots \otimes \widehat{Q}_1(\lambda_{r-1})^{(r-1)}.$$
Since the TMC holds for $G_1$, this module also carries the $G$-structure
$$T(\hat{\lambda}_0) \otimes T(\hat{\lambda}_1)^{(1)} \otimes \cdots \otimes T(\hat{\lambda}_{r-1})^{(r-1)},$$
which a repeated application of \cite[Lemma E.9]{rags} shows to be isomorphic to the tilting module $T(2(p^r-1)\rho + w_0\lambda)$.  Since its highest weight is in $(p^r-1)\rho+X_+$, it is projective over $G_r$ \cite[Lemma E.8]{rags}, and thus is a lift to $G$ of $\widehat{Q}_r(\lambda)$.  This proves that the TMC holds for $G_r$.

Conversely, if the TMC fails for $G_1$, then there is a weight $\lambda \in X_1$ such that $T(\hat{\lambda})$ is of larger dimension than $Q_1(\lambda)$ (which must split off as a $G_1$-summand).  For any $r>1$, the indecomposability of the tilting module $\St_{r-1}$ over $G_{r-1}$ implies that
$$\St_{r-1} \otimes T(\hat{\lambda})^{(r-1)} \cong T((p^{r-1}-1)\rho+(p^{r-1})\hat{\lambda})=T(2(p^r-1)\rho+w_0((p^{(r-1)}-1)\rho+\lambda)).$$
Over $G_{r-1}T$, we have an isomorphism
$$\widehat{Q}_r((p^{(r-1)}-1)\rho+p^{r-1}\lambda) \cong \St_{r-1} \otimes \widehat{Q}_1(\lambda)^{(r-1)}.$$
It now follows by a dimension argument that
$$T(2(p^r-1)\rho+w_0((p^{(r-1)}-1)\rho+\lambda)) \not\cong \widehat{Q}_r((p^{(r-1)}-1)\rho+p^{r-1}\lambda).$$
Thus, the TMC fails for $G_{r}$ when $r > 1$ as well.
\end{proof}

\subsection{Induced modules and good (p,r)-filtrations} For a dominant weight $\mu$, a long-standing question asks  when $\nabla(\mu)$ admits a good $(p,r)$-filtration.  We will see later (Theorem \ref{T:summary}), that an affirmative answer to this question can provide a means of verifying the TMC.  

One potential approach to the good $(p,r)$-filtration question is to make use of the structure of baby Verma modules and the fact that
$$\nabla(\mu) = \ind_{B}^{G} \mu = \ind_{G_rB}^G\circ\ind_{B}^{G_rB} \mu = \ind_{G_rB}^G  \widehat{Z}_r'(\mu).$$

For any $\mu \in X_+$, $\widehat{Z}_r'(\mu)$ admits a composition series  with simple $G_rB$-compositions factors of the form $L(\s_0)\otimes p^r\s_1$ for $\s_0 \in X_r$ and $\s_1 \in X$.  With any such composition factor, note that (cf.\cite[II.9.13(2)]{rags})
$$
\ind_{G_rB}^G(L(\s_0)\otimes p^r\s_1) \cong L(\s_0)\otimes \ind_{G_rB}^G(p^r\s_1) \cong L(\s_0)\otimes [\ind_{B}^G\s_1]^{(r)} = L(\s_0)\otimes \nabla(\s_1)^{(r)}.
$$

The following proposition (stated in a very general context) gives several different conditions on the terms of such a composition series to show that $\nabla(\mu)$ admits a good $(p,r)$-filtration.  The first of these conditions is certainly well-known and encompasses the conditions given in  \cite[Prop. II.9.14]{rags}.

\begin{theorem}\label{T:induction} Let $M$ be a finite-dimensional $B$-module and let 
$$0=N_{0}\subseteq N_{1}\subseteq N_{2} \subseteq \dots   \subseteq N_{t}=\operatorname{ind}_{B}^{G_{r}B}M$$
be a composition series as a $G_{r}B$-module for $\operatorname{ind}_{B}^{G_{r}B}M$ with 
$N_{i}/N_{i-1}\cong L(\mu_{i})\otimes p^{r}\sigma_{i}$ where $\mu_{i}\in X_{r}$ and $\sigma_{i}\in X$ for $i=1,2,\dots,t$. 
Suppose that one of the following conditions holds
\begin{itemize}
\item[(a)] for all $1 \leq j \leq t$, $R^1\ind_B^G\s_j = 0$;
\item[(b)]  for all $1\leq j < i \leq t$ with $\mu_i = \mu_j$, 
$$\operatorname{Hom}_{G}(\nabla(\s_i),R^1\ind_{B}^{G}\s_j) = 0;$$ 
\item[(c)] both of the following hold:
\begin{itemize}
\item[(i)] for all $1 \leq i \leq t$, $R^2\ind_{B}^{G}\s_i= 0$, and
\item[(ii)] for all $1\leq j < i \leq t$ with $\mu_i = \mu_j$, either
$$\operatorname{Hom}_{G}(\nabla(\s_i),R^1\ind_{B}^{G}\s_j) = 0$$ 
or any non-zero homomorphism $\nabla(\s_i) \to R^1\ind_{B}^{G}\s_j$ is an isomorphism.
\end{itemize}
\end{itemize}
Then $\operatorname{ind}_{B}^{G}M$ has a good $(p,r)$-filtration. 
\end{theorem}

\begin{proof}  Recall the definition of $\nabla^{(p,r)}(\lambda)$ from Section~\ref{S:Notation}. 
We first make a general observation: for any $i, j$, 
\begin{equation}\label{E:hom}
\begin{split}
\Hom_{G}&\left(\nabla^{(p,r)}(\mu_i + p^r\s_i),L(\mu_j)\otimes[R^1\ind_B^G\s_j]^{(r)}\right)     \\
&= \Hom_{G}\left(L(\mu_i)\otimes \nabla(\s_i)^{(r)},L(\mu_j)\otimes [R^1\ind_B^G\s_j]^{(r)}\right) \\
&\cong \Hom_{G/G_r}\left(\nabla(\s_i)^{(r)},\Hom_{G_r}(L(\mu_i),L(\mu_j))\otimes[R^1\ind_B^G\s_j]^{(r)}\right) \neq 0
\end{split}
\end{equation}
implies that $\mu_i = \mu_j$ (since $\mu_i, \mu_j \in X_r$) and $\Hom_{G}(\nabla(\s_i),R^1\ind_{B}^{G}\s_j) \neq 0$.  In particular, the $i$th and $j$th composition factors must have the same isotypic component.

If condition (a) holds, the long exact sequence in induction could be used to prove the claim.   However, since condition (a) immediately implies condition (b), we simply provide a proof of the claim under condition (b).  To that end, assume condition (b) holds.  We will first show that $\operatorname{Hom}_{G}(\nabla^{(p,r)}(\mu_{i}+p^{r}\sigma_{i}),R^{1}\operatorname{ind}_{G_{r}B}^{G} N_{k})=0$ for all $i$ and $k$ via induction on $k$.  
For $k=1$, one has 
$$R^{1}\text{ind}_{G_{r}B}^{G} N_{1}\cong R^{1}\text{ind}_{G_{r}B}^{G}\left( L(\mu_{1})\otimes p^{r}\sigma_{1}\right) \cong L(\mu_{1})\otimes [R^{1}\text{ind}_{B}^{G} \sigma_{1}]^{(r)},$$ 
so the claim follows by using the hypothesis and \eqref{E:hom}. 

Now assume that $\operatorname{Hom}_{G}(\nabla^{(p,r)}(\mu_{i}+p^{r}\sigma_{i}),R^{1}\operatorname{ind}_{G_{r}B}^{G} N_{k-1})=0$ for all $i$. 
Consider the short exact sequence
 $$0\rightarrow N_{k-1}\rightarrow N_{k} \rightarrow N_{k}/N_{k-1} \rightarrow 0.$$
One has a long exact sequence 
$$0\rightarrow \text{ind}_{G_{r}B}^{G} N_{k-1} \rightarrow  \text{ind}_{G_{r}B}^{G} N_{k} \rightarrow  \text{ind}_{G_{r}B}^{G} N_{k}/N_{k-1} \overset{\phi}{\rightarrow}  R^{1}\text{ind}_{G_{r}B}^{G} N_{k-1} \to \cdots.$$ 
By the induction hypothesis the map $\phi$ is zero, therefore, one has an exact sequence
$$0\rightarrow R^{1}\text{ind}_{G_{r}B}^{G} N_{k-1} \rightarrow R^{1}\text{ind}_{G_{r}B}^{G} N_{k} \rightarrow R^{1}\text{ind}_{G_{r}B}^{G} N_{k}/N_{k-1}. $$ 
Now by the induction hypothesis, the hypothesis of the proposition, and \eqref{E:hom},
$$\operatorname{Hom}_{G}(\nabla^{(p,r)}(\mu_{i}+p^{r}\sigma_{i}),  R^{1}\text{ind}_{G_{r}B}^{G} N_{k-1})=0$$ 
and 
$$\operatorname{Hom}_{G}(\nabla^{(p,r)}(\mu_{i}+p^{r}\sigma_{i}), R^{1}\text{ind}_{G_{r}B}^{G} N_{k}/N_{k-1})=0$$ 
for all $i$.
Consequently, 
$$\operatorname{Hom}_{G}(\nabla^{(p,r)}(\mu_{i}+p^{r}\sigma_{i}),  R^{1}\text{ind}_{G_{r}B}^{G}N_{k})=0$$
for all $i$.  

We will next show that $\text{ind}_{G_{r}B}^{G} N_{k}$ has a good $(p,r)$-filtration for all $k$. In particular, this will show that 
$\text{ind}_{B}^{G}N_{t}=\text{ind}_{G_{r}B}^{G} \circ\text{ind}_{B}^{G_{r}}M=\text{ind}_{B}^{G}M$ has a good $(p,r)$-filtration. 
For $k=1$, one has 
$$\text{ind}_{G_{r}B}^{G} N_{1}\cong \text{ind}_{G_{r}B}^{G}\left( L(\mu_{1})\otimes p^{r}\sigma_{1}\right) \cong L(\mu_{1})\otimes [\text{ind}_{B}^{G} \sigma_{1}]^{(r)},$$ 
which verifies the claim. 

Now assume that $\text{ind}_{G_{r}B}^{G}N_{k-1}$ has a good $(p,r)$-filtration and consider the short exact sequence, 
 $$0\rightarrow N_{k-1}\rightarrow N_{k} \rightarrow N_{k}/N_{k-1} \rightarrow 0.$$
One has a long exact sequence 
$$0\rightarrow \text{ind}_{G_{r}B}^{G} N_{k-1} \rightarrow  \text{ind}_{G_{r}B}^{G} N_{k} \rightarrow  \text{ind}_{G_{r}B}^{G} N_{k}/N_{k-1} \overset{\phi}{\rightarrow}  R^{1}\text{ind}_{G_{r}B}^{G} N_{k-1}\to \cdots.$$ 
Recall that $\text{ind}_{G_{r}B}^{G} N_{k}/N_{k-1} \cong \nabla^{(p,r)}(\mu_{k}+p^{r}\sigma_{k})$.
By the first part, $\operatorname{Hom}_{G}(\nabla^{(p,r)}(\mu_{i}+p^{r}\sigma_{i}),  R^{1}\text{ind}_{G_{r}B}^{G} N_{k-1})=0$ for all $i$, so the map $\phi$ is zero. 
Hence, $\text{ind}_{G_{r}B}^{G}N_{k}$ has a good $(p,r)$-filtration. 

Lastly, consider scenario (c).  We first show by induction on $k$ that $R^2\ind_{G_rB}^{G}N_k = 0$ for all $1 \leq k \leq t$.  This holds by assumption for $k = 1$.   For arbitrary $k$, consider the short exact sequence 
$$0\rightarrow N_{k-1}\rightarrow N_{k} \rightarrow N_{k}/N_{k-1} \rightarrow 0$$ and the associated long exact sequence
$$\cdots \to R^2\ind_{G_rB}^{G}N_{k-1} \to R^2\ind_{G_rB}^GN_k \to R^2\ind_{G_rB}^{G}N_k/N_{k-1} \to \cdots.$$
Since the rightmost term is zero by assumption and the leftmost term is zero by the inductive hypothesis, the middle term is also zero as claimed.

We now proceed to show inductively on $k$ that $\ind_{G_rB}^GN_k$ has a good $(p,r)$-filtration for all $k$. To do so, we simultaneously show inductively that 
$R^1\ind_{G_rB}^GN_k$ has a filtration with factors of the form $R^1\ind_{G_rB}^GN_j/N_{j-1} \cong L(\mu_j)\otimes[ R^1\ind_{B}^G\s_j]^{(r)}$ 
for $1 \leq j \leq k$ (not necessarily all $j$ appearing).  Both claims trivally hold for $k = 1$, so we now proceed inductively.  Consider again the short exact sequence
$$0\rightarrow N_{k-1}\rightarrow N_{k} \rightarrow N_{k}/N_{k-1} \rightarrow 0$$ and the beginning of the associated long exact sequence
\begin{align*}
0 &\to \ind_{G_rB}^GN_{k-1} \to \ind_{G_rB}^GN_k \to \ind_{G_rB}^GN_k/N_{k-1}  \\
	&\overset{\phi}{\to} R^1\ind_{G_rB}^GN_{k-1} \to R^1\ind_{G_rB}^GN_k \to R^1\ind_{G_rB}^GN_k/N_{k-1} \to 0,
\end{align*}
where the culminating zero term is due to the vanishing of $R^2$ shown above.  By induction, we may assume that $\ind_{G_rB}^GN_{k-1}$ admits a good $(p,r)$-filtration and $R^1\ind_{G_rB}N_{k-1}$ admits a filtration with factors of the form  $R^1\ind_{G_rB}^GN_j/N_{j-1} \cong L(\mu_j)\otimes[ R^1\ind_{B}^G\s_j]^{(r)}$ 
for some $1 \leq j \leq k-1$.

If the map $\phi$ is zero, then exactness immediately gives the claims for $\ind_{G_rB}^GN_k$ and $R^1\ind_{G_rB}^GN_k$.   Suppose now that $\phi$ is non-zero.  Then $\ind_{G_rB}^GN_k/N_{k-1} \cong \nabla^{(p,r)}(\mu_k + p^r\s_k)$ is necessarily non-zero.  By the inductive assumption on $R^1\ind_{G_rB}^GN_{k-1}$, the original assumption (ii), and \eqref{E:hom}, the map $\phi$ must be an injection, from which it follows that $\ind_{G_rB}^GN_k \cong \ind_{G_rB}^GN_{k-1}$ and,  hence, admits a good $(p,r)$-filtration. Furthermore, $\phi$ must be an isomorphism onto one of the $R^1\ind_{G_rB}^GN_j/N_{j-1}$ factors.   Thus exactness gives the claimed filtration on $R^1\ind_{G_rB}^GN_k$.

\end{proof}

In our applications of interest, it will suffice to consider the case $r = 1$.  The following gives a weight condition that can be used to verify that $R^1\ind_{B}^G \s_1 = 0$ and potentially apply the preceding proposition. By $\a_0$ and $\tilde{\a}$ we denote the highest short root and the highest long root of the root system, respectively.

\begin{prop}\label{P:weights} Let $\mu, \s_0 \in X_1$ and $\s_1 \in X.$ If $L(\s_0) \otimes p\s_1$ is a $G_1B$-composition factor of $\widehat{Z}_1'((p-1)\rho +\mu)$ and 
$R^1\ind_{B}^G \s_1 \neq 0$,
then there exists a simple root $\alpha_i$ and a weight $\gamma$ in the weight lattice of $\St_1$ such that
\begin{equation}\label{E:weights}
\langle \gamma, \alpha_i^{\vee} \rangle \leq -2p -\langle \s_0, \tilde{\alpha}^{\vee} \rangle- \langle \mu, \alpha_i^{\vee} \rangle \leq -2p.
\end{equation}  
\end{prop}

\begin{proof}
  If $L_1(\s_0) \otimes p\s_1$ is a composition factor of $\widehat{Z}_1'((p-1)\rho +\mu)$ with $R^{1}\ind_{B}^G \s_1 \neq 0$, then there  exists a simple root $\alpha_i$ with $\langle \s_1, \alpha_i^{\vee} \rangle \leq -2$
  (cf. \cite[Prop. II.5.4]{rags}). 
 Furthermore, any weight of the form $w\s_0+p\s_1$, with $w \in W,$ also appears in the weight lattice of  $\widehat{Z}_1'((p-1)\rho +\mu).$ We may choose $w$ such that $-w^{-1}\alpha_i \in \{ \a_0, \tilde{\a}\}.$ Then 
$ \langle w\s_0, \alpha_i^{\vee} \rangle \leq -\langle \s_0, \tilde{\a}^{\vee} \rangle.$ 
The weight lattice of $\widehat{Z}_1'((p-1)\rho +\mu)$ is obtained from the weight lattice of $\St_1$ by simply adding $\mu$ to each weight. There has to exist a $\gamma$ in the weight lattice of $\St_1$ such that $\gamma  = w\s_0+p\s_1-\mu.$ 
Taking the inner product with $\alpha_i^{\vee}$ yields the assertion.
\end{proof}

\subsection{Translation and (p,r)-filtrations}  Let $T_{\nu}^{\nu'}$ denote the translation functor associated to weights $\nu, \nu' \in \overline{C}_{\mathbb Z}$ (closure of the bottom alcove). 
In cases when $p\geq h$,  we prove below that for a given $G$ it is sufficient to prove the TMC when $r=1$ for regular restricted weights. 

\begin{prop} If $p\geq h,$ then it is sufficient to verify the TMC for $r=1$ and all $\la \in X_1 \cap W_p \cdot 0$.
\end{prop} 

\begin{proof}  Assume that $p \geq h$ and that $T(\hat{\la})|_{G_1T} \cong \widehat{Q}_1(\la)$ for all  $\la \in X_1 \cap W_p \cdot 0.$  Given $\ga \in X_1$ there exist unique $\mu \in \overline{C}_{\mathbb Z}$ and $w \in W_p$ such that $\ga = w \cdot \mu$ and $w \cdot 0$ is minimal among all $wx\cdot 0$ with $x \in \text{Stab}_{W_p}(\mu).$  In addition,  we denote by $\hat{w}$ the  element of $W_p$ defined via $\hat{w} \cdot 0 = 2(p-1)\rho + w_0(w\cdot 0)=p\cdot 2\rho +(w_0w)\cdot 0.$ Then $\hat{w} \cdot \mu =2(p-1)\rho+w_0(w\cdot \mu)= 2(p-1)\rho+w_0\gamma= \hat{\gamma}.$ Moreover, 
$w_0w.0$ will be maximal among all $w_0 w x\cdot 0$ with $x \in \text{Stab}_{W_p}(\mu),$ as will be $\hat{w} \cdot 0$ among all $\hat{w}x\cdot 0.$

By \cite[Proposition 5.2]{And00}, $T_{\mu}^{0} T(\hat{\gamma}) = T_{\mu}^{0} T(\hat{w}\cdot \mu) =T(\hat{w} \cdot 0).$ Similarly, it follows from \cite[II 11.10]{rags} that $T_{\mu}^0 \widehat{Q}_1(\gamma) = T_{\mu}^0 \widehat{Q}_1(w\cdot \mu) =\widehat{Q}_1(w \cdot 0).$ 
We know that $\hat{Q}_1(\gamma)$ is a $G_1T$-summand of $T(\hat \gamma).$ 
Therefore, we can write $T(\hat{\gamma})|_{G_1T} \cong \widehat{Q}_1(\gamma) \oplus M,$ where $M$ is either zero or an injective $G_1T$-module.  But now $M\neq 0$ implies $T_{\mu}^0 M \neq 0.$  This would contradict our assumption that $T(\hat{w} \cdot 0)|_{G_1T}\cong \widehat{Q}_1(w \cdot 0).$ Hence, $M=0$ and  $T(\hat{\gamma})|_{G_1T} \cong \widehat{Q}_1(\gamma).$

\end{proof}
One of the main methods to verify the TMC is to show that (i) for each $\lambda \in X_r$, $\St_{r} \otimes 
L(\lambda)$ has a good filtration and (ii)  $\nabla^{(p,r)}(\hat{\lambda})$ has a good $(p,r)$-filtration (see Theorem~\ref{T:summary} and \cite[Theorem 1.2.1]{So}). 

Recall that good and Weyl filtrations behave well under translation (cf. \cite[II.7.13]{rags}).  For any $\nu, \nu' \in \overline{C}_{\mathbb Z}$, if $M$ has a good or Weyl filtration, then so does $T_{\nu}^{\nu'}(M)$. 
One would like this to be true for good $(p,r)$-filtrations or Weyl $(p,r)$-filtrations in order to verify (ii).  In the context of Weyl $(p,r)$-filtrations, it was observed (without formal proof) in \cite{PS} that this holds in some cases. 
We provide a proof of this statement here.

\begin{prop}\label{P:translation} Let $M$ be a rational $G$-module and  $\nu, \nu' \in \overline{C}_{\mathbb Z}$. Assume that $p \geq h$.  If $M$ has a good $(p,r)$-filtration (Weyl $(p,r)$-filtration) such that each factor $\nabla^{(p,r)}(\la)$ ($\Delta^{(p,r)}(\la)$, respectively) has $\lambda$ being $p$-regular. Then $T_{\nu}^{\nu'}(M)$ admits a good $(p,r)$-filtration (Weyl $(p,r)$-filtration, respectively).
\end{prop}

\begin{proof}  As observed in \cite[Section 5]{PS}, for a $p$-regular weight $\lambda$, $T_{\nu}^{\nu'}(\Delta^{(p,r)}(\la))$ is either zero or another $\Delta^{(p,r)}(\mu)$ for some $\mu \in X^+$.   We provide the details of the equivalent claim in the dual case of $\nabla^{(p,r)}(\la)$.  To do so, we also need to work in the category of $G_rB$-modules, where we have the simple module $\widehat{L}_r(\sigma) = L(\sigma)|_{G_rB}$ for a $\sigma\in X_{r}$ and the $G_rB$-translation functor $\widehat{T}_{\nu}^{\nu'}$. Write $\lambda = \lambda_0 + p^{r}\lambda_1$ for $\lambda_0 \in X_r$.  Then we have
\begin{align*}
\nabla^{(p,r)}(\lambda) &= \nabla^{(p,r)}(\lambda_0 + p\lambda_1) = L(\lambda_0)\otimes\nabla(\lambda_1)^{(r)} = L(\lambda_0)\otimes [\operatorname{ind}_{B}^{G}\lambda_1]^{(r)} \\
	&\cong \operatorname{ind}_{G_rB}^{G}(L(\lambda_0)|_{G_rB}\otimes p^{r}\lambda_1) \quad \text{(\cite[II.9.13(3)]{rags})} \\
	&= \operatorname{ind}_{G_rB}^{G}(\widehat{L}_r(\lambda_0)\otimes p^{r}\lambda_1) \\
	&\cong  \operatorname{ind}_{G_rB}^{G}(\widehat{L}_r(\lambda_0 + p^{r}\lambda_1)) \quad \text{(\cite[II.9.6(6)]{rags})}.
\end{align*}
From \cite[Lemma 3.1(b)]{CPS09}, one has
$$
T_{\nu}^{\nu'}\left(\nabla^{(p,r)}(\lambda)\right) = T_{\nu}^{\nu'}\left(\operatorname{ind}_{G_rB}^G\left(\widehat{L}_r(\lambda_0 + p^{r}\lambda_1)\right)\right) 
	= \operatorname{ind}_{G_rB}^{G}\left(\widehat{T}_{\nu}^{\nu'}\left(\widehat{L}_r(\lambda_0 + p^{r}\lambda_1)\right)\right).
$$
Under the assumption that $\lambda$ is $p$-regular, $\widehat{T}_{\nu}^{\nu'}$ takes a simple to another simple or zero (cf. \cite[II.9.22]{rags}.  So, reversing the argument above, one obtains some $\nabla^{(p,r)}(\mu)$ or zero. Since translation functors are exact, the filtration claims follow.
\end{proof}


\section{New General Techniques Involving Truncated Categories} \label{S:weightThms}

\subsection{}\label{S:JantzenGeneralization}

In this section we give a generalization of Jantzen's classic result on lifting of the $G_1$-PIMs when $p \ge 2h-2$.  This method of proof is different. It is what one might call ``weight-based," in that it primarily relies upon weight combinatorics.  Such methods break down for smaller primes, as Jantzen notes.  Nonetheless, in at least one specific small prime case dealt with in this paper, this general method is useful, and it also gives important information when dealing with higher Frobenius kernels.

For a weight $\tau \in X_+$, we denote by $\textup{Mod}(\tau)$ the truncated subcategory of all finite-dimensional rational $G$-modules whose highest weights are less than or equal to $\tau$.

\begin{prop}\label{P:injproj}
For $r \geq 1,$ let $\la_0, \sigma \in X_r$ and $\la_1 \in X_+$ with $\la = \la_0 + p^r \la_1 \leq \sigma$. Suppose one of the following conditions holds.
\begin{itemize} 
\item[(a)]  Assume that all weights $\eta \in X_+$ satisfy 
\begin{equation}\label{E:injproj}
\text{If } p^r \eta \leq -w_0\sigma + \la_0, \text{ then }\Ext_G^1(L(-w_0\la_1), L(\eta)) =0.
\end{equation}
\item[(b)]  Let $d$ be the smallest integer such that $\langle \eta, \alpha_0^{\vee}\rangle=d$ and $\Ext_G^1(L(-w_0\la_1), L(\eta)) \ne 0$ for $\eta \in X_+$.  Assume
$$p^r\cdot d> \langle-w_0\sigma + \la_0,\a_0^{\vee}\rangle.$$
\end{itemize} 
Then $\St_r \otimes L(\la)$ is both injective and projective in $\textup{Mod}((p^r-1)\rho+\sigma)$.
\end{prop}


\begin{proof} (a) Let $\mu \in X_+$ be such that $\mu \le (p^r-1)\rho + \sigma$.  In order to prove the injectivity, it suffices to show that 
$\Ext_G^1(L(\mu),\St_r \otimes L(\la))\cong \Ext_G^1(L(\mu),\St_r \otimes L(\la_0)\otimes L(\la_1)^{(r)}) $ vanishes. Note that
$$ \Ext_G^1(L(\mu),\St_r \otimes L(\la_0)\otimes L(\la_1)^{(r)})  \cong\Ext_G^1(\St_r \otimes L(-w_0\la_1)^{(r)}, L(\mu)^* \otimes L(\la_0))$$
by \cite[I.4.4]{rags} and the self-duality of $\St_r$.  Applying \cite[II.10.4]{rags}, the only composition factors of $L(\mu)^* \otimes L(\la_0)$ that can be extended by $\St_r \otimes L(-w_0\la_1)^{(r)}$ must have the form $\St_r \otimes L(\eta)^{(r)}$. Using the five term exact sequence in the Lyndon-Hochschild-Serre spectral sequence, we have
\begin{eqnarray*}
\Ext_G^1(\St_r\otimes L(-w_0\la_1)^{(r)}, \St_r \otimes L(\eta)^{(r)}) &\cong& \Ext_{G/G_{r}}^1( L(-w_0\la_1)^{(r)}, L(\eta)^{(r)})\\
&\cong& \Ext_{G}^1( L(-w_0\la_1), L(\eta)). 
\end{eqnarray*}
The simple module $\St_r \otimes L(\eta)^{(r)}$ is a composition factor of $L(\mu)^* \otimes L(\la_0)$, thus $(p^r - 1)\rho + p^r\eta \leq -w_0\mu + \lambda_0$.  
Since $\mu \le (p^r-1)\rho+\sigma$, any such $\eta$ must satisfy
$$(p^r-1)\rho+p^r\eta \le (p^r-1)\rho-w_0\sigma + \la_0$$
which implies  $$p^r\eta \le -w_0\sigma + \la_0.$$
Now Condition (\ref{E:injproj}) guarantees vanishing of  $\Ext_G^1(L(-w_0\la_1), L(\eta))$ and hence also of $\Ext_G^1(L(\mu),\St_r \otimes L(\la))$.
Therefore, $\St_r \otimes L(\la)$ is an injective object in $\textup{Mod}((p^r-1)\rho+\la)$.  By the isomorphism
$$\Ext_G^1(\St_r \otimes L(\la),L(\mu)) \cong \Ext_G^1(L(\mu),\St_r \otimes L(\la))=0,$$
it follows that $\St_r \otimes L(\la)$ is also a projective object in $\textup{Mod}((p^r-1)\rho+\sigma)$.

(b) Let $d$ be the smallest integer such that $\langle \eta, \alpha_0^{\vee}\rangle=d$ and $\Ext_G^1(L(-w_0\la_1, L(\eta)) \ne 0$ for $\eta \in X_+$. 
The assumption $p^r\cdot d> \langle-w_0\sigma + \la_0,\a_0^{\vee}\rangle$ forces  $p^r \eta \nleq -w_0\sigma + \la_0$, and the result follows from part (a). 
\end{proof}

\subsection{} Observe that if $\St_r\otimes L(\la)$ is both injective and projective in $\textup{Mod}((p^r-1)\rho+\la)$, then it is necessarily tilting.  To see this, recall the general fact that, for a $G$-module $M$, if $\Ext_G^1(\Delta(\mu),M) \neq 0$, then $M$ must have a composition factor $L(\ga)$ with $\ga > \mu$.   Consider $\Ext_G^1(\Delta(\mu),\St_r\otimes L(\la))$ for an arbitrary dominant weight $\mu$.   Since the highest weight of $\St_r\otimes L(\la)$ is $(p^r-1)\rho + \la$, by the aforementioned observation, this Ext-group can only be nonzero if $\Delta(\mu) \in \textup{Mod}((p^r-1)\rho+\la)$.  However, in that case, the Ext-group vanishes by the injectivity assumption.   Hence, $\St_r\otimes L(\la)$ has a good filtration. By a dual argument with $\nabla(\mu)$, one sees that it also admits a Weyl filtration and hence is tilting.  
The next result shows that under certain injectivity/projectivity conditions in the truncated category, one can explicitly identify tilting modules restricted over $G_rT$ with specific projective indecomposable 
modules. 

\begin{prop}\label{P:IfIPTMC}
For $r \geq 1,$ let $\la_0, \sigma \in X_r$ and $\la_1 \in X_+$ with $\la = \la_0 + p^r \la_1 \leq \sigma$.   If  $\St_r \otimes L(\la_0)$  is injective and projective in $\textup{Mod}((p^r-1)\rho+\la_0)$ and $\St_r \otimes L(\la)$ is injective and projective in $\textup{Mod}((p^r-1)\rho+\sigma)$, then
$$T((p^r-1)\rho+\la) \mid_{G_rT} \cong \widehat{Q}_r((p^r-1)\rho+w_0\la)\otimes L(\la_1)^{(r)}$$
as a $G_rT$-module.
\end{prop}

\begin{proof}
Note that $\la \leq \sigma$ implies that $\St_r \otimes L(\la)$ is also injective and projective in 
$\textup{Mod}((p^r-1)\rho+\la).$ We will first discuss the case $\la_1=0.$
Let $I$ denote the injective hull of $L((p^r-1)\rho+w_0\la)$ in the truncated category $\textup{Mod}((p^r-1)\rho+\la)$.  Observe that
$$\Hom_G(L((p^r-1)\rho+w_0\la_0),\St_r \otimes L(\la_0)) \cong \Hom_{G_r}(L((p^r-1)\rho+w_0\la),\St_r \otimes L(\la_0)) \cong k.$$
Therefore, the injective $I$ appears exactly once as an indecomposable $G$-summand of $\St_r\otimes L(\la_0)$.  If one views $I$ as a $G_r$-module, it is a direct sum of some $Q_r(\gamma)$s.  
However, the $G$ and $G_r$-socle of $I$ are both simple.  It follows that $I \cong Q_r((p^r-1)\rho+w_0\la_0)$.  But, since $\St_r \otimes L(\la_0)$ is tilting, the injective hull $I$ is also an indecomposable tilting module.  It contains the highest weight $(p^r-1)\rho+\la_0$.  Hence, $Q_r((p^r-1)\rho+w_0\la_0) \cong I \cong T((p^r-1)\rho+\la_0)$.

If $\la_1\neq 0$ one can write $\St_r \otimes L(\la) \cong \St_r \otimes L(\la_0) \otimes L(\la_1)^{(r)}.$ From the earlier discussion, it follows that the tensor product has  $T((p^r-1)\rho+\la_0) \otimes L(\la_1)^{(r)}$ as a $G$-summand. This module (being a summand of a tilting module) is tilting and has simple $G$-socle $L(\la).$ It is therefore isomorphic to $I$ and as a $G_rT$-module isomorphic to  $\widehat{Q}_r((p^r-1)\rho+w_0\la)\otimes L(\la_1)^{(r)}$.
\end{proof}

\subsection{TMC Criterion and Jantzen's result} Combining the propositions from the preceding subsections, we obtain a criterion for the validity of the Tilting Module Conjecture.  

\begin{theorem}\label{T:pairingbound} Let $G$ be a semisimple, simply connected algebraic group scheme defined and split over ${\mathbb F}_{p}$. Suppose that for 
every $\eta \in X_+$ with $\textup{Ext}_G^1(k,L(\eta)) \ne 0$, one has  
$$p\langle \eta, \alpha_0^{\vee} \rangle > 2(p-1)(h-1).$$ 
Then 
$$T((p-1)\rho + \la)|_{G_{1}T}\cong \widehat{Q}_1((p-1)\rho+w_0\la)$$ 
for all $\la \in X_1$, and the TMC holds for $G$.
\end{theorem}

\begin{proof}   Given $\la \in X_1$, apply Proposition~\ref{P:injproj}(b) with $\la_1=0$ and $\la_0=\sigma \in X_1$. The first claim then follows from Proposition~\ref{P:IfIPTMC}, and hence the TMC holds for $r = 1$.
As observed in Proposition \ref{P:equivTMCr=1}, if the TMC holds for $r = 1$, then it holds for all $r$. 
\end{proof}

With the preceding theorem, we can now recover Jantzen's theorem that verifies the TMC for $p\geq 2h-2$.  

\begin{cor} \label{C:JantzenCor}
If $p \ge 2h-2$, then the TMC holds for $G$.
\end{cor}

\begin{proof} Suppose $\eta \in X_{+}$ with $\Ext^1_{G}(k,L(\eta)) \neq 0$.  As shown in \cite[Proposition 2.3.1(b)]{KN}, one always has $\langle\eta,\a_0^{\vee}\rangle \ge 2(p-h+1)$, so that
$$
p\langle\eta,\a_0^{\vee}\rangle \geq 2p(p-h + 1) \geq 2p(2h -2 -h + 1) = 2p(h-1) > 2(p-1)(h-1)
$$
and the claim follows from Theorem~\ref{T:pairingbound}.
\end{proof}

\subsection{\bf General  Observation for $p=2h-3$} For this section set $r=1$. 
We start out with two general observations for any group $G$ having the property that $2h-3$ is a prime.

\begin{lemma}\label{L:2h-31} Let $p=2h-3.$ If $0 \neq \mu \in X_+ \cap W_{p}\cdot 0$, then $(h-2)\alpha_0 \leq \mu$ and 
$2(h-2) \leq \langle \mu, \alpha_0^{\vee} \rangle.$ 
\end{lemma} 
\begin{proof}
Any dominant non-zero weight $\mu$  in $X_+ \cap W_{p}\cdot 0$ satisfies $$\mu \geq s_{\alpha_0,p}\cdot 0 = (p-(h-1))\alpha_0= (h-2)\alpha_0.$$
The second statement follows by taking the inner product of both sides with $\alpha_{0}$.
\end{proof}
\begin{lemma}\label{L:2h-32} Let $p=2h-3,$  $\la_0, \sigma \in X_1$, and $\la_1 \in X_+$ with $\la = \la_0 + p \la_1 \leq \sigma.$ If  $ \langle \la + \sigma, \alpha_0^{\vee} \rangle < 2p(h-2)$ 
then 
\begin{itemize}
\item[(a)]
$\St_1 \otimes L(\la)$  is projective and injective in $\textup{Mod}((p-1)\rho+\sigma).$

\item[(b)] $T((p-1)\rho+\la) \cong T((p-1)\rho+\la_0) \otimes L(\la_1)^{(1)} $ is the injective hull and projective cover of $L((p-1)\rho+w_0\la_0) \otimes L(\la_1)^{(1)}$ in $\textup{Mod}((p-1)\rho+\sigma).$

\item[(c)]  $T((p-1)\rho+\la)|_{G_{1}T}\cong \widehat{Q}_1((p-1)\rho+w_0 \la)  \otimes L(\la_1)^{(1)}$  as a $G_1T$-module.
\end{itemize}
\end{lemma} 
\begin{proof} We will first discuss the case $\la_1=0.$
We will make use of Proposition \ref{P:injproj} and its notation, particularly the integer $d$.  Lemma~\ref{L:2h-31} implies that $d = 2(h-2)$, and we obtain
$$\langle \la - w_0\sigma, \alpha_0^{\vee}  \rangle = \langle \la + \sigma, \alpha_0^{\vee} \rangle < 2p(h-2)=p\cdot d.$$ 
It follows immediately from Proposition \ref{P:injproj} and Proposition \ref{P:IfIPTMC} that 
$\St_1 \otimes L(\la) $ is injective and projective in $\textup{Mod}((p-1)\rho+\sigma)$ and
$T((p-1)\rho+\la) \cong \widehat{Q}_1((p-1)\rho+\omega_0\la) $ as a $G_1T$-module.

Next assume that $\la_1 \neq 0.$ First note that $\la  \leq \sigma$ and  $ \langle \la + \sigma, \alpha_0^{\vee} \rangle < 2p(h-2)$ clearly implies that 
$ \langle 2 \la_0 , \alpha_0^{\vee} \rangle < 2p(h-2).$ From the above case we may conclude that $\St_1 \otimes L(\la_0)$ is injective and projective in $\textup{Mod}((p-1)\rho+\la_0).$

Now observe that
$\la = \la_0 + p \la_1 \leq \sigma$ and   $ \langle \la + \sigma, \alpha_0^{\vee} \rangle < 2p(h-2)$ implies 
$\langle p\la_1, \alpha_0^{\vee}  \rangle \leq \langle \la , \alpha_0^{\vee} \rangle < p(h-2).$  It follows that 
$\langle \la_1 + \rho, \alpha_0^{\vee}  \rangle < 2h-3=p.$  Hence, the weight $\la_1$ is contained in the lowest dominant alcove. 
Therefore,  any dominant weight $\eta$ linked to $\la_1$ is greater than or equal to $s_{\alpha_0,p}\cdot \la_1.$ Observe  that 
$\langle \la_1 +s_{\alpha_0,p}\cdot \la_1+ 2\rho, \alpha_0^{\vee}  \rangle=  2p,$ which implies  
$$\langle  \eta , \alpha_0^{\vee} \rangle \geq \langle s_{\alpha_0,p}\cdot \la_1, \alpha_0^{\vee}  \rangle 
	= 2p - 2(h-1) - \langle\la_1,\a_0^{\vee}\rangle
	=  2(h-2) - \langle  \lambda_1 , \alpha_0^{\vee}  \rangle.$$
However,
$$\langle \sigma + \lambda_0 , \alpha_0^{\vee}  \rangle= \langle \sigma + \lambda , \alpha_0^{\vee}  \rangle - p\langle  \lambda_1 , \alpha_0^{\vee}  \rangle < 2p(h-2)- p\langle  \lambda_1 , \alpha_0^{\vee}  \rangle.$$
Hence, $p\langle\eta,\a_0^{\vee}\rangle > \langle\s + \la_0,\a_0^{\vee}\rangle$, and so
the claim follows from Proposition~\ref{P:injproj} and Proposition~\ref{P:IfIPTMC}.
\end{proof}

We can now state a result that shows that the TMC holds for most weights when $p=2h-3$. 

\begin{theorem}\label{T:2h-31} Let $p=2h-3.$ If $\mu \in X_1$ and $\mu$ is not contained in $\overline{C}_{\mathbb Z}$ (i.e., the closure of the lowest alcove) then $T(2(p-1)\rho+ w_0\mu)|_{G_{1}T}\cong 
\widehat{Q}_1(\mu)$ (as a $G_1T$-module).
\end{theorem}
\begin{proof} Note that any restricted weight $\mu$ that is not contained in the closure of the lowest alcove satisfies $\langle \mu + \rho, \alpha_0^{\vee} \rangle > p = 2h-3.$ Therefore, 
$\langle \mu, \alpha_0^{\vee} \rangle > h-2$ and  
$$\langle (p-1)\rho +w_0\mu, \alpha_0^{\vee} \rangle < (p-1)(h-1)- (h-2)= p(h-2).$$ 
The assertion follows from the statements of Lemma \ref{L:2h-32} by setting $\la=\sigma=(p-1)\rho+w_0\mu.$ 
\end{proof} 


\section{Splitting Conditions That Are Equivalent To The TMC}

In this section we prove new results about module splittings which are equivalent to the TMC.  This provides the primary tools to determine what happens in small characteristic.  We give a summary of related conditions that imply the Tilting Module Conjecture in Theorem \ref{T:summary}.  

\subsection{The Main Idea}\label{S:Idea}

Up to a non-zero scalar, there are unique $G$-module homomorphisms
$$L(\lambda) \hookrightarrow T(\hat{\lambda}), \qquad T(\hat{\lambda}) \twoheadrightarrow L(\lambda).$$
We may tensor each of these maps with $\St_r$, and it was shown in \cite{So} that the truth of the TMC for a fixed $r\geq 1$ is equivalent to knowing that the $G$-module homomorphism
$$\St_r \otimes L(\lambda) \hookrightarrow \St_r \otimes T(\hat{\lambda})$$
splits (in $G$-mod) for all $\lambda \in X_r$.  Such a splitting is equivalent to saying that the modules
$$\St_r \otimes L(\lambda) \quad \text{and} \quad \St_r \otimes \left(T(\hat{\lambda})/L(\lambda)\right)$$
are both tilting modules.

We now summarize the work in \cite{So} to clearly explain why the property that these modules are tilting (stated in the Overarching Assumption (OA)) 
yields an inductive argument proving the Tilting Module Conjecture. 

\bigskip
\noindent\textbf{Overarching Assumption (OA):} For each $\lambda \in X_r$, the modules
$$\St_r \otimes L(\lambda) \quad \text{and} \quad \St_r \otimes \left(T(\hat{\lambda})/L(\lambda)\right)$$
are tilting.
\vskip .25cm 
\noindent 
We will now present the inductive argument, which works by decreasing induction on restricted weights.  We recall that the standard partial order on weights can be refined to the partial order $\le_{\mathbb{Q}}$ where $\lambda \le_{\mathbb{Q}} \mu$ if $\mu-\lambda$ is a linear combination of roots having non-negative rational coefficients.

\bigskip
\noindent\textbf{Base Case:} If $\la = (p^{r}-1)\rho$, then $\hat{\la} = (p^{r}-1)\rho$, $T(\hat{\la}) \mid_{G_{r}T}= \St_r$, and $\qq(\la) = \St_{r}$. Hence, $T(\hat{\la})\mid_{G_{r}} \cong Q_{r}(\la)$.

\bigskip
\noindent\textbf{Induction Hypothesis:} For $\lambda \in X_r$, if $\mu \in X_r$ and $\mu >_{\mathbb{Q}} \lambda$, then $T(\hat{\mu})\mid_{G_{r}T} \cong \widehat{Q}_r(\mu)$.

\bigskip
\noindent\textbf{Inductive Step:} The points that follow show that $G_r$-summands of $T(\hat{\lambda})$ corresponding to the isotypic components of $\soc_{G_r}\, T(\hat{\lambda})$ can be realized as summands over $G$.  Since $T(\hat{\lambda})$ is indecomposable over $G$, it will then follow that $T(\hat{\lambda})\mid_{G_rT}\cong \widehat{Q}_r(\lambda)$.
\begin{enumerate}
\item If $\lambda \ne \mu \in X_r$ is such that $\text{Hom}_{G_r}(L(\mu),T(\hat{\lambda})) \ne 0$, then, by \cite[Prop. 4.1.3]{So}, $\mu >_{\mathbb{Q}} \lambda$.
\item By induction hypothesis, $T(\hat{\mu}) \mid_{G_rT}\cong \widehat{Q}_r(\mu)$.
\item Since $\St_r \otimes L(\mu)$ is tilting (by OA), the module $\text{Hom}_{G_r}(L(\mu),T(\hat{\lambda}))^{(-r)}$ is tilting \cite[Lemma 3.1.2]{So}.
\item Since $\St_r \otimes \left(T(\hat{\mu})/L(\mu)\right)$ is tilting (by OA), the $G$-module injection
$$L(\mu) \otimes \text{Hom}_{G_r}(L(\mu),T(\hat{\lambda})) \hookrightarrow T(\hat{\lambda})$$
extends to a $G$-module homomorphism 
$$T(\hat{\mu}) \otimes \text{Hom}_{G_r}(L(\mu),T(\hat{\lambda})) \rightarrow T(\hat{\lambda}),$$
as follows from the vanishing of the extension group
$$\Ext_G^1((T(\hat{\mu})/L(\mu)) \otimes \text{Hom}_{G_r}(L(\mu),T(\hat{\lambda})),T(\hat{\lambda})) =0$$
(see proof of \cite[Theorem 5.1.1]{So}).
\item Since $T(\hat{\mu})\mid_{G_rT} \cong \widehat{Q}_r(\mu)$, this extended homomorphism is also injective.
\item Applying the $\tau$-functor of Section \ref{S:Notation}, we get a $G$-module surjection
$$T(\hat{\lambda}) \twoheadrightarrow T(\hat{\mu}) \otimes \text{Hom}_{G_r}(L(\mu),T(\hat{\lambda})).$$
\item The composite of the last two maps is an isomorphism over $G_rT$, hence also over $G$.  This composition means that the indecomposable $G$-module $T(\hat{\lambda})$ splits over $G$, a contradiction.
\item Conclusion is that if $\lambda \ne \mu \in X_r$, then $\text{Hom}_{G_r}(L(\mu),T(\hat{\lambda})) = 0$, thus $T(\hat{\lambda})\mid_{G_rT}\cong \widehat{Q}_r(\lambda)$.
\end{enumerate}

\subsection{Splitting Conditions} We begin with an observation involving the evaluation map $\varepsilon:\nabla(\hat{\lambda}) \rightarrow \widehat{Z}_1^{\prime}(\hat{\lambda})$ for 
$\lambda\in X_{r}$ that does not hold in general when $\hat{\lambda}$ is replaced by an arbitrary dominant weight. 

\begin{prop}\label{P:nablaz}
If $\lambda \in X_r$, then the evaluation map $\varepsilon:\nabla(\hat{\lambda}) \rightarrow \widehat{Z}_r^{\prime}(\hat{\lambda})$ is surjective.
\end{prop}

\begin{proof} The $G$-module $T(\hat{\lambda})$ is projective upon restriction to $G_rT$ (cf. \cite[Lemma II.E.8]{rags}) and hence over $B_r$.  So, by \cite[Prop. II.11.2]{rags}, $T(\hat{\lambda})$ has a $G_rB$-filtration by modules of the form $\widehat{Z}_r^{\prime}(\mu)$.  Furthermore, by \cite[II.11.2 Rem. (4)]{rags} there is some such filtration in which the final factor is $\widehat{Z}_r^{\prime}(\hat{\lambda})$, implying the existence of a surjective $G_rB$-homomorphism
$$T(\hat{\lambda}) \rightarrow \widehat{Z}_r^{\prime}(\hat{\lambda}).$$
Since $\nabla(\hat{\lambda}) \cong \textup{ind}_{G_rB}^G \, \widehat{Z}_r^{\prime}(\hat{\lambda})$, by \cite[Prop. I.3.4 (b)]{rags} this morphism factors as
$$T(\hat{\lambda}) \rightarrow \nabla(\hat{\lambda}) \xrightarrow{\varepsilon} \widehat{Z}_r^{\prime}(\hat{\lambda}).$$
Since the composition of morphisms is surjective, it follows that $\varepsilon$ is also.
\end{proof}

Let $\lambda \in X_r$.  We have the following module inclusions as  $G_rB$-homomorphisms:
$$L(\lambda) \hookrightarrow \widehat{Z}_r^{\prime}(\lambda) \hookrightarrow T(\hat{\lambda}).$$
Similarly, we have a chain of surjective homomorphisms
$$T(\hat{\lambda}) \twoheadrightarrow \nabla(\hat{\lambda}) \twoheadrightarrow \widehat{Z}_r^{\prime}(\hat{\lambda}) \twoheadrightarrow L(\lambda).$$
Each inclusion map (resp. surjective map) is unique up to scalar multiple.  We may tensor the modules in these homomorphisms with $\St_r$ (and extend to homomorphisms between these tensors in the obvious way).  Focusing on the second series of homomorphisms we get
$$\St_r \otimes T(\hat{\lambda}) \twoheadrightarrow \St_r \otimes \nabla(\hat{\lambda}) \twoheadrightarrow \St_r \otimes \widehat{Z}_r^{\prime}(\hat{\lambda}) \twoheadrightarrow \St_r \otimes L(\lambda).$$

Set $R(\lambda)$ to be the module defined by the short exact sequence 
$$0 \rightarrow R(\lambda) \rightarrow \nabla(\hat{\lambda}) \rightarrow L(\lambda) \rightarrow 0.$$
One also obtains a short exact sequence 
\begin{equation}\label{eq:SESsteinberg}
0 \rightarrow \St_{r}\otimes R(\lambda) \rightarrow \St_r\otimes \nabla(\hat{\lambda}) \rightarrow \St_r\otimes L(\lambda) \rightarrow 0.
\end{equation}
We can now relate the splitting of the surjective map in (\ref{eq:SESsteinberg}) to the existence of good filtrations. 

\begin{prop}\label{P:splitgood}
Let $\lambda \in X_r$.  The following are equivalent:
\begin{enumerate}
\item[(a)] The canonical surjection of $G$-modules $\St_r \otimes \nabla(\hat{\lambda}) \twoheadrightarrow \St_r \otimes L(\lambda)$ splits.
\item[(b)] $\St_r \otimes R(\lambda)$ has a good filtration.
\end{enumerate}
\end{prop}

\begin{proof} We consider the short exact sequence (\ref{eq:SESsteinberg}). Since $\St_{r}\otimes\nabla(\hat{\la})$ has a good filtration, if the sequence splits, then each summand admits a good filtration. So (a) implies (b). Conversely, if $\St_{r}\otimes R(\la)$ has a good filtration, then, by \cite[Cor. II.4.17]{rags}, so does $\St_r\otimes L(\la)$.  The module $\St_r\otimes L(\la)$ is invariant under the $\tau$-functor, which sends modules with good filtrations to those with Weyl filtrations (and vice versa). Hence, $\St_r\otimes L(\la)$ has a Weyl filtration, and therefore is tilting.   As there are no non-trival extensions of a module with a Weyl filtration by one with a good filtration, the sequence must split.  
\end{proof}

One way to establish statement (b), hence also obtaining the splitting in (a), is to verify the stronger hypothesis of the following proposition, which is a special case of a more general result observed by Andersen \cite{And01}.  Namely, if $\St_r \otimes L(\mu)$ is tilting for all $\mu \in X_r$ and $M$ admits a good $(p,r)$-filtration, then $\St_r\otimes M$ admits a good filtration.  As Andersen's result was stated specifically for $r=1$, we include a short proof for general $r$.

\begin{prop}\label{P:goodpgood}
If $R(\lambda)$ has a good $(p,r)$-filtration and $\St_r \otimes L(\mu)$ is tilting for all $\mu \in X_r$, then $\St_r \otimes R(\lambda)$ has a good filtration. 
\end{prop}

\begin{proof}
If suffices to prove that $\St_r \otimes (L(\mu) \otimes \nabla(\mu^{\prime})^{(r)})$ has a good filtration for all $\mu \in X_r$ and $\mu^{\prime} \in X^+$.  We have that $\St_r$ is a summand of $\St_r \otimes ({\St_r}^* \otimes \St_r)$, and ${\St_r}^* \cong \St_r$.  Thus
$$\St_r \otimes (L(\mu) \otimes \nabla(\mu^{\prime})^{(r)})$$
can be realized as a summand of
$$N = \St_r \otimes (\St_r \otimes L(\mu)) \otimes (\St_r \otimes \nabla(\mu^{\prime})^{(r)}).$$
We have $\St_r \otimes \nabla(\mu^{\prime})^{(r)} \cong \nabla((p^r-1)\rho+p^r\mu^{\prime})$, and by assumption $\St_r \otimes L(\mu)$ has a good filtration, hence $N$ has a good filtration.  Each summand of $N$ therefore has a good filtration, proving the claim.
\end{proof}

\begin{remark}
The hypothesis that $\St_r \otimes L(\mu)$ is tilting for all restricted weights $\mu$ is known to hold in a number of cases (cf. \cite{BNPS}). 
\end{remark}

\subsection{} The following proposition shows that the splitting for the map given in Proposition~\ref{P:splitgood}(a) is equivalent to splitting for analogous maps with tilting and baby Verma modules. 

\begin{prop}\label{P:equivalent}
Let $\lambda \in X_r$.  The following are equivalent:
\begin{enumerate}
\item[(a)] The canonical surjection of $G$-modules $\St_r \otimes T(\hat{\lambda}) \twoheadrightarrow \St_r \otimes L(\lambda)$ splits.
\item[(b)] The canonical surjection of $G$-modules $\St_r \otimes \nabla(\hat{\lambda}) \twoheadrightarrow \St_r \otimes L(\lambda)$ splits.
\item[(c)] The canonical surjection of $G_rB$-modules $\St_r \otimes \widehat{Z}_r^{\prime}(\hat{\lambda}) \twoheadrightarrow \St_r \otimes L(\lambda)$ splits.
\end{enumerate}
\end{prop}

\begin{proof}
The surjection in (a) factors through the map in (b) which factors through the map in (c), so we get from this that (a) $\Rightarrow$ (b) $\Rightarrow$ (c).  If (c) holds, then we obtain (b) by applying the functor $\ind_{G_rB}^{G}(-)$ to the split sequence of $G_rB$-modules, yielding the desired split sequence of $G$-modules.  

Suppose that (b) holds.   Then $\St_r \otimes L(\lambda)$ has a good filtration, and by $\tau$-duality is tilting.  Applying Proposition \ref{P:splitgood}, we also have that $\St_r \otimes R(\lambda)$ has a good filtration. 
 Let $N$ be the kernel of the surjective map $T(\hat{\lambda}) \rightarrow L(\lambda)$.  Since $\nabla(\hat{\lambda})$ is the final good filtration factor in $T(\hat{\lambda})$, the fact that 
$\St_r \otimes R(\lambda)$ has a good filtration, as does $\St_r \otimes \nabla(\mu)$ for all other good filtration factors $\nabla(\mu)$ of $T(\hat{\lambda})$, implies that $\St_r \otimes N$ has a good filtration.  
Since $\St_r \otimes L(\lambda)$ is tilting, this then implies that the sequence
$$0 \rightarrow \St_r \otimes N \rightarrow \St_r \otimes T(\hat{\lambda}) \rightarrow \St_r \otimes L(\lambda) \rightarrow 0$$
splits, proving (a).
\end{proof} 

In \cite[Section 2.2]{BNPS2}, the authors described the interrelationships and hierarchies between the Donkin's Tilting Module  and $(p,r)$-Filtration Conjectures. 
We are now ready to demonstrate how the validity of these conjectures imply the splitting of the maps given in Proposition~\ref{P:equivalent}. 

\begin{theorem}\label{T:pfilttosplit} Let $\lambda\in X_r$. Suppose that 
\begin{enumerate}
\item[(a)] $\St_r \otimes L(\sigma)$ is tilting for all $\sigma\in X_r$; 
\item[(b)] $T(\hat{\lambda}) \mid_{G_{r}T}\cong \widehat{Q}_r(\lambda)$; and
\item[(c)]  either $\nabla(\hat{\lambda})$ or $T(\hat{\lambda})$ have a good $(p,r)$-filtration. 
\end{enumerate} 
Then the canonical surjection of $G$-modules $\St_r \otimes T(\hat{\lambda}) \twoheadrightarrow \St_r \otimes L(\lambda)$ splits.
\end{theorem}

 \begin{proof} If $T(\hat{\lambda})\mid_{G_{r}T} \cong \widehat{Q}_r(\lambda)$, then the $G_r$-head of $T(\hat{\lambda})$ and of $\nabla(\hat{\lambda})$ are both simple, isomorphic to $L(\lambda)$.  Suppose that $\nabla(\hat{\lambda})$ has a good $(p,r)$-filtration.  Then the final factor of this filtration must be $L(\lambda)$, hence $R(\lambda)$ has a good $(p,r)$-filtration.  The splitting then follows from Propositions \ref{P:goodpgood}, \ref{P:splitgood}, and \ref{P:equivalent}.

If $T(\hat{\lambda})$ has a good $(p,r)$-filtration, then in a similar way it follows that its $G_r$-radical must have a good $(p,r)$-filtration.  Analogous arguments as those used in the aforementioned propositions just cited yield the proof here.
\end{proof}

\subsection{}  The following theorem provides key conditions that may be used to verify the Tilting Module Conjecture.  A key aspect of this theorem (distinguishing it from \cite[Theorem 1.2.1]{So}) is that it allows one to potentially verify different conditions for different $p^r$-restricted weights.

\begin{theorem}\label{T:summary}
Let $G$ be a semisimple, simply connected algebraic group scheme defined and split over ${\mathbb F}_{p}$, and suppose it holds that $\St_r\otimes L(\mu)$ is a tilting module for every $\mu \in X_r$.  The Tilting Module Conjecture holds for $G$ and for a fixed $r\geq 1$  if for each $\mu \in X_r$, at least one of the following conditions is also true (some of these are equivalent):
\begin{itemize}
\item[(a)] $\St_r \otimes T(\hat{\mu}) \rightarrow \St_r \otimes L(\mu)$ splits.
\item[(b)] $\St_r \otimes \nabla(\hat{\mu}) \rightarrow \St_r \otimes L(\mu)$ splits.
\item[(c)] $\nabla(\hat{\mu})$ has a good $(p,r)$-filtration.
\item[(d)] $T(\hat{\mu})$ has a good $(p,r)$-filtration.
\end{itemize}
\end{theorem}

\begin{proof} 
The inductive step outlined in Section \ref{S:Idea} says that for $\lambda \in X_r$, we have $T(\hat{\lambda})\mid_{G_{r}T} \cong \widehat{Q}_r(\lambda)$ if for every $p^{r}$-restricted $\mu >_{\mathbb{Q}} \lambda$ we have both that $T(\hat{\mu})\mid_{G_rT} \cong \widehat{Q}_1(\mu)$, and that the canonical map $\St_r \otimes T(\hat{\mu}) \rightarrow \St_r \otimes L(\mu)$ splits.  By Proposition \ref{P:equivalent}, conditions (a) and (b) are equivalent, and by Theorem \ref{T:pfilttosplit}, if $T(\hat{\mu})\mid_{G_{r}T}\cong \widehat{Q}_1(\mu)$, then (c) and (d) imply (a).  Thus if the hypothesis of this theorem is true and $T(\hat{\mu})\mid_{G_rT}\cong \widehat{Q}_r(\mu)$, we immediately have that $\St_r \otimes T(\hat{\mu}) \rightarrow \St_r \otimes L(\mu)$ splits.  This guarantees that the inductive step always holds, so that 
$T(\hat{\lambda})\mid_{G_rT} \cong \widehat{Q}_r(\lambda)$ for all $\lambda \in X_r$.
\end{proof}

It should be mentioned that one gets two weights ``for free'' when checking these conditions.

\begin{prop}\label{P:special}
For all $p$, we have that
$$\St_r \otimes T(\hat{\lambda}) \rightarrow \St_r \otimes L(\lambda)$$
splits for $\lambda = 0,$ $(p^r-1)\rho$. 
\end{prop}

\begin{proof}
If $\lambda = (p^r-1)\rho$, then $\hat{\lambda}=(p-1)\rho$.  So $L(\la) \cong \St_r \cong T(\hat{\lambda})$, and it is immediate that $\St_r \otimes \St_r\rightarrow \St_r \otimes \St_r$ splits. Parts (c) and (d) of Theorem \ref{T:summary} are also immediately evident.

On the other hand, $\hat{0}=2(p^r-1)\rho$ and the question is whether $\St_r\otimes T(\hat{0}) \to \St_r\otimes k \cong \St_r$ splits.  The canonical map $\St_r\otimes\St_r\otimes\St_r \to \St_r$ splits.  Since the unique trivial submodule of $\St_r \otimes \St_r$ is contained in $T(\hat{0})$, this surjection factors through
$$\St_r\otimes\St_r\otimes\St_r \to \St_r \otimes T(\hat{0}) \to \St_r,$$
proving the claim.
\end{proof}


\section{Type $\rm{A}_n$, $n\leq 3$} 

\subsection{} Recall from Propostion \ref{P:equivTMCr=1} that to verify the Tilting Module Conjecture, it suffices to consider the case $r = 1$. In subsequent sections this fact will be used when 
we verify the TMC for other low rank cases.  For $\Phi=\rm{A}_{n}$, $n\leq 3$, by \cite{KN} or \cite{BNPS}, it is known that $\St_{1}\otimes L(\mu)$ admits a good filtration for all $\mu \in X_1$. 
Therefore, by Theorem \ref{T:summary}, in order to verify the TMC in these cases it suffices to prove that $\nabla(\hat{\lambda})$ has a good 
$p$-filtration for all $\lambda\in X_{1}$.

\subsection{\bf Types $\rm{A}_{1}$ and $\rm{A}_{2}$} The TMC was verified earlier by others and is known to hold for all primes.  In the first case, as $h = 2$, this follows from using the 
general $p \geq 2h - 2$ bound.  In the latter case, $h = 3$, so the TMC holds for $p > 3$ using the aforementioned bound. A proof is given for $p=2,3$ in \cite[II.11.16]{rags}. Moreover, 
Donkin \cite{Don} showed that the $G$-liftings on the $G_1$ projective indecomposable modules are unique.     

We observe that our methods provide a quick alternate proof for the type $\rm{A}_2$, $p = 2, 3$ cases.  In order to apply Theorem~\ref{T:summary}, we verify that $\nabla(\hat{\la})$ admits a good $p$-filtration for each $\la \in X_1$.  To that end, identifying  $\nabla(\hat{\la}) = \ind_{B}^{G}\hat{\la} =  \ind_{G_1B}^G \widehat{Z}_1'(\hat{\la}) $, it suffices to show that one of the conditions in Theorem~\ref{T:induction} holds.  Let $L(\s_0)\otimes p\s_1$ be a $G_1B$-composition factor of $\widehat{Z}_1'(\hat{\la})$.  We claim that $R^1\ind_{B}^G\s_1 = 0$ (which would complete this argument).   On the contrary, if $R^1\ind_{B}^G \s_1 \neq 0$, by Proposition \ref{P:weights}, there is a weight $\ga$ of $\St_1$ with $\langle\ga,\a_i^{\vee}\rangle \leq -2p$.   However, since $\langle \rho,\a_0^{\vee}\rangle = 2$, one has $-2(p-1) \leq \langle\ga,\a_i^\vee\rangle \leq 2(p-1)$, thus giving a contradiction.  

\subsection{\bf Type  $\rm{A}_{3}$} In the case when $\Phi=\rm{A}_{3}$, one has $h=4$. Therefore, the TMC holds for $p\geq 2h-2=6$ by Corollary~\ref{C:JantzenCor}. It remains to verify the TMC for $p=2$, $3$ and 
$5$. 

\subsection{\bf Type $\rm{A}_3$, $p=2$} We use the same argument as in the type $\rm{A}_2$ case to show that all $\nabla(\hat{\la}),$ $\la \in X_1,$ have a  good $p$-filtration. 
Since $\langle \rho, \alpha_0^{\vee}\rangle = 3$, one concludes that all weights $\gamma$ of $\St_1$ satisfy $-2p= -4<-3 \leq \langle \gamma, \alpha_i^{\vee}\rangle$ for any simple root $\alpha_i.$ It follows that no weight of $\St_1$ satisfies the inequality \eqref{E:weights}, and the claim follows.

\subsection{\bf Type $\rm{A}_3$, $p=3$} Again we show that $\nabla(\hat{\la})$ has a good $p$-filtration for all $\la \in X_1$.  Since $\langle \rho, \alpha_0^{\vee}\rangle = 3$, one concludes that all weights $\gamma$ of $\St_1$ satisfy $-2p=-6 \leq \langle \gamma, \alpha_i^{\vee}\rangle$ for any simple root $\alpha_i.$ Assume now that 
$L(\s_0) \otimes p\s_1$ is a $G_1B$-composition factor of $\widehat{Z}_1'((p-1)\rho +\mu)$ and 
$R^1\ind_{B}^G\s_1 \neq 0.$ Equation \eqref{E:weights} immediately implies that there exists an $\alpha_i$ with $ \langle \gamma, \alpha_i^{\vee}\rangle= -6$ and $\langle\mu,\alpha_i^{\vee} \rangle =0.$ Moreover, $\s_0=0.$ The following table contains all the weights $\gamma$ in the weight lattice of $\St_1$ with $ \langle \gamma, \alpha_i^{\vee}\rangle=-6$ (each appears with multiplicity one), the unique weight $\mu \in X_1$ such that $\gamma +\mu = p \s_1 
 \in pX, $  the weight $\s_1,$  and the highest weight $\hat{\la}=(p-1)\rho + \mu$ of $\nabla(\hat{\la})$.  

\bigskip
\begin{tabular}{ | c | c | c | c | c |}
\hline
$\gamma$ & $\mu$ &$ \s_1$ &$s_{\alpha_i} \cdot \s_1$&  $(p-1)\rho+ \mu$\\
\hline 
(-6,2,2) & (0,1,1) & (-2,1,1) & (0,0,1)& (2,3,3)\\
\hline
(2,2,-6) & (1,1,0) & (1,1,-2) &(1,0,0)& (3,3,2)\\
\hline

(4,-6,4) & (2,0,2) & (2,-2,2) &(1,0,1)& (4,2,4)\\
\hline
(-6,4,-2) & (0,2,2) & (-2,2,0) & (0,1,0)& (2,4,4)\\
\hline
(-2,4,-6) & (2,2,0) & (0,2,-2) & (0,1,0)& (4,4,2)\\
\hline
(2,-6,2) & (1,0,1) & (1,-2,1) &(0,0,0) & (3,2,3)\\
\hline
\hline
(0,3,-6) & (0,0,0) & (0,1,-2) &(0,0,0)& (2,2,2)\\
\hline
(-6,3,0) & (0,0,0) & (-2,1,0) &(0,0,0)& (2,2,2)\\
\hline
(3,-6,3) & (0,0,0) & (1,-2,1) &(0,0,0)& (2,2,2)\\
\hline
\end{tabular}
\bigskip

The last three entries are not of interest because $\nabla(2,2,2)$ is just the Steinberg module, which, being simple, admits a good $p$-filtration. Consider Theorem~\ref{T:induction}.  
The only situation of concern is if there appears another composition factor $L(0) \otimes p\eta$ with $(p-1)\rho +\mu \geq p\eta >p s_{\alpha_i} \cdot \s_1$ such that there is a non-zero and non-isomorphic map
from $\nabla(\eta)$ to $R^1\ind_{B}^G\s_1 = \nabla(s_{\alpha_i} \cdot \s_1)$.  Note that in each case of the table $\nabla(s_{\a_i}\cdot \s_1) \cong  L(s_{\a_i}\cdot \s_1)$.   Furthermore, $\eta$ is potentially problematic only if $\nabla(\eta) \not\cong L(\eta)$. From linkage information, one can see that this fails  in all but the first two (symmetric) cases. Note that $\nabla(1,1,0)/L(1,1,0) \cong L(0,0,1)$  and that $(3,3,0) < (2,3,3).$  By symmetry, it remains to show that $\nabla(2,3,3)$ has a good $p$-filtration. We will prove the equivalent statement that $\Delta(2,3,3)$ has a $p$-Weyl filtration.

The formal characters of all simple modules with restricted highest weight were first determined in \cite{Jan74}. They can also be obtained via the tables of \cite{L}. The highest weights of the $G$-composition factors of $\Delta(2,3,3)$ can therefore be calculated. They are
$ (2, 3, 3 )$, $( 3, 1, 4 ),$ $( 2, 4, 1 ),$ $ ( 1, 2, 4 ),$ 
  $( 4, 0, 3),$ $( 3, 3, 0 ),$ $( 1, 4, 0 ),$ $( 4, 1, 1 ),$ 
 $ ( 0, 2, 3),$ $ ( 0, 3, 1),$ $ ( 3, 0, 2 ),$
  $ ( 5, 0, 0 ),$ 
  $( 1, 1, 2 ),$ $ ( 0, 0, 3 ),$ $ ( 0, 1, 1),$ and  $ ( 1, 0, 0 ). $ All appear with multiplicity one, except for $(5,0,0)$, $(0,3,1)$, $(0,0,3)$ and $(1,0,0),$ which appear with multiplicities $2,$ $2$, $3$ and $2$, respectively. 
  
  We want to directly show that a $p$-Weyl filtration can be constructed.  
  Consider a composition factor $L(\mu_0)\otimes L(\mu_1)^{(1)}$ of $\Delta(2,3,3)$.   If $L(\mu_1) = \Delta(\mu_1)$ for all $\mu_1$, then the composition series would immediately give a $p$-Weyl filtration.
  However, there are exactly two composition factors with $L(\mu_1) \neq \Delta(\mu_1).$ These have highest weights  $(2,3,3)$ and $(3,3,0),$ respectively. 
 In the following, we construct a non-zero homomorphism $\phi:\Delta(3,3,0) \to  \Delta(2,3,3). $ It is then shown that the image of $\phi,$ denoted by $S,$ and the cokernel of $\phi,$ denoted by $Q,$ have $p$-Weyl filtrations, thus producing the desired filtration for $\Delta(2,3,3).$ 
  
   Observe that in the above list of highest weights for the composition factors of $\Delta(2,3,3)$ only $(2,3,3)$, $(3,1,4)$ and $(2,4,1)$ are strictly greater than $(3,3,0).$ 
  To verify the existence of $\phi$ we make use of the Jantzen filtration for the Weyl module $\Delta(2,3,3).$ Given
  $$\Delta(2,3,3)=V^0\supseteq V^1 \supseteq V^2 \supseteq  \cdots,$$ with $V^i$ as defined in \cite[II.8.19]{rags}, 
 set $SF = \sum_{i>0} \text{ch } V^i.$ A straightforward calculation shows that the characters of $L(3,1,4)$ and $L(2,4,1)$ appear exactly once in $SF$ while $\text{ch }L(3,3,0)$ appears twice. Note that $L(3,3,0)$ is a composition factor of both $\Delta(3,1,4)$ and $\Delta(2,4,1)$ whose characters appear once in the sum formula $SF$. Since the multiplicity of $L(3,3,0)$ in $\Delta(2,3,3)$ is one, we conclude that $L(3,3,0)$ appears in $V^2$ while $L(3,1,4)$ and $L(2,4,1)$ do not. Therefore, $(3,3,0)$ is a maximal weight of $V^2$ and one obtains the desired non-trivial homomorphism  $\phi:\Delta(3,3,0) \to V^2\hookrightarrow \Delta(2,3,3).$ 
 
Next we show that $S$ has a $p$-Weyl filtration. The composition factors of  $\Delta (3,3,0)$ have highest weights 
\begin{equation}\label{E:comp}
(3, 3, 0), ( 1, 4, 0 ), ( 4, 1, 1 ), ( 0, 3, 1), ( 3, 0, 2 ), ( 1, 1, 2 ), ( 0, 0, 3), (1, 0, 0 ).
\end{equation}
As discussed above, only the weight $(3,3,0)$ is potentially problematic.   We show that there exists a short exact sequence
\begin{equation}\label{E:ses1}
0 \to S_1 \to S \to S_2 \to 0 
\end{equation} 
where $S_2 \cong \Delta(1,1,0)^{(1)}$
and $S_1$ (is zero or) has composition factors whose highest weights come from \eqref{E:comp}, except $(3,3,0)$.  Hence, both $S_1$ (if non-zero) and $S_2$ would admit $p$-Weyl filtrations.
To this end, first note that there is a projection of $\Delta (3,3,0)$ onto $\Delta(1,1,0)^{(1)}$ and the radical of $\Delta(1,1,0)$ is just $L(0,0,1).$ 
This yields a $G$-extension between $L(3,3,0)$  and $L(0,0,3).$ We claim that the only composition factor $L(\mu)$  of  $\Delta(3,3,0)$ with $\Ext_{G}^1(L(3,3,0), L(\mu))\neq 0$ is $L(0,0,3).$
Consider the Lyndon-Hochschild-Serre spectral sequence:
$$
E_2^{i,j} = \Ext_{G/G_1}^i(L(3,3,0),\Ext_{G_1}^j(k,L(\mu))) \Rightarrow \Ext_{G}^{i + j}(L(3,3,0),L(\mu)).
$$
For $\mu \neq (0,0,3)$, we have
\begin{align*}
\Ext_{G}^{1}(L(3,3,0),L(\mu))&\cong \Hom_{G/G_1}\left(L(3,3,0),\Ext_{G_1}^{1}(k,L(\mu_0)\otimes L(\mu_1)^{(1)})\right) \\
	&\cong \Hom_{G}(L(1,1,0),\Ext_{G_1}^{1}(k,L(\mu_0))^{(-1)}\otimes L(\mu_1)).
\end{align*}

The following table lists all restricted weights $\mu_0$ with $\Ext_{G_1}^1(k, L(\mu_0)) \neq 0$ (cf. \cite{Jan74, Jan91}).  Note that the first entry is the only non-trivial $G$-extension with a restricted weight. The other extensions arise via $\Ext_{G_1}^1(k,\nabla(p\omega_i-\alpha_i))$, where $\omega_i$ denotes a fundamental weight and $\a_i$ the associated simple root.

\bigskip
\begin{tabular}{ | c  | c | c | c | c |}
\hline
$\mu_0$ & $(0,2,0)$ &$(1,1,0)$&$(0,1,1)$&$(1,1,1)$\\
\hline
$\Ext_{G_1}^1(k, L(\mu_0))^{(-1)}$ & $k$& $L(1,0,0)$&$L(0,0,1)$&$L(0,1,0)$\\
\hline
\end{tabular}
\bigskip

From the table we conclude that  we only have to rule out the weights $(1,4,0)$ and $ ( 4, 1, 1 )$ from \eqref{E:comp}. Note that
\begin{eqnarray*} 
\Ext_{G}^1(L(3,3,0), L(1,4,0)) &\cong& \Hom_{G}\left(L(1,1,0), \Ext_{G_1}(k,L(1,1,0))^{(-1)} \otimes L(0,1,0)\right)\\
&\cong&  \Hom_{G}(L(1,1,0), L(1,0,0) \otimes L(0,1,0)).
\end{eqnarray*}
But $ L(1,0,0) \otimes L(0,1,0)$ is the indecomposable tilting module with highest weight $(1,1,0)$ and simple socle $L(0,0,1).$ Therefore, the $\Ext$-group vanishes.
Similarly,
\begin{eqnarray*}
\Ext_{G}^1(L(3,3,0), L(4,1,1)) &\cong& \Hom_{G}\left(L(1,1,0), \Ext_{G_1}(k,L(1,1,1))^{(-1)} \otimes L(1,0,0)\right)\\
&\cong&  \Hom_{G}(L(1,1,0), L(0,1,0) \otimes L(1,0,0))=0.
\end{eqnarray*}
We conclude that the second radical layer of $\Delta(3,3,0)$ is simple and isomorphic to $L(0,0,3).$ This implies that the  short exact sequence \eqref{E:ses1}  indeed exists, unless $S =L(3,3,0).$ However,
\begin{eqnarray*}
 \Hom_G(L(3,3,0), \Delta(2,3,3))& \hookrightarrow& \Hom_G\left(L(1,1,0)^{(1)}, \St_1 \otimes L(0,1,1)\right)\\
&\cong& 
 \Hom_G(\St_1 \otimes L(1,1,0)^{(1)},  L(0,1,1))\\
 &\cong& 
 \Hom_G(L(5,5,2),  L(0,1,1))
 =0.
\end{eqnarray*}
 Therefore, since $S_1$ (if non-zero) and $S_2$ have $p$-Weyl filtrations, so does $S.$

 Finally, observe also that the weight $(5,0,0)$ does not appear in (\ref{E:comp}).  Therefore, $L(5,0,0)$ is not a composition factor of $S$.  Recall that $\Delta(0,1,1)$ has two composition factors: $L(0,1,1)$ and $L(1,0,0)$. Therefore, the surjection of $\Delta (2,3,3)$ onto $L(2,0,0) \otimes \Delta(0,1,1)^{(1)}$ yields a surjection of the cokernel  $Q=\Delta (2,3,3)/S$ onto  $L(2,0,0) \otimes \Delta(0,1,1)^{(1)}.$ 
 We obtain a short exact sequence 
 $$0 \to Q_1 \to Q \to L(2,0,0) \otimes \Delta(0,1,1)^{(1)} \to 0.$$
The highest weights of the composition factors of $Q_1$ include neither $(3,3,0)$ nor $(2,3,3).$ This implies that $Q_1$ and $Q$ also afford  $p$-Weyl filtrations.  Hence,  $\Delta(2,3,3)$ has a $p$-Weyl filtration and the assertion follows.

\subsection{\bf Type $\rm{A}_3$, $p=5$} 
Since $p=2h-3$, it follows from Theorem \ref{T:2h-31} that the verification of the  TMC has been reduced to weights in the closure of the lowest alcove. Note that the highest  root $\a_0= (1,0,1)$ lies on the alcove wall separating the lowest alcove  from the second lowest alcove. Our first goal is to verify the TMC for the weight $2(p-1)\rho + w_0\a_0 = 2(p-1)\rho - \a_0$. We begin with a weight estimate that will be used 
later in this section. 

 \begin{lemma}\label{L:A3p5} Let $G$ be of type  $\rm{A}_3$ with $p=5$.  If $\la \in X_+$ with $\la  \neq (p-1)\rho-\a_0$ and $(p-1)\rho + \la \uparrow 2(p-1)\rho -\a_0,$ then $\langle \la, \alpha_0^{\vee} \rangle < p(h-2).$
\end{lemma}

\begin{proof} If $\la \neq (p-1)\rho -\a_0$ and $(p-1)\rho + \la \uparrow 2(p-1)\rho -\a_0$ then there exists a positive integer $m,$ a sequence of reflections $s_1, s_2, ... , s_m \in W_p,$ and weights $\mu_1, \mu_2, ... , \mu_{m-1} \in X^+$ with 
$$2(p-1)\rho - \a_0 >  s_1 \cdot (2(p-1)\rho - \a_0) =\mu_{1} >  s_2\cdot \mu_1 = \mu_2 > \cdots >  s_m\cdot \mu_{m-1}=(p-1)\rho + \la.$$
Assume that $\langle \la, \alpha_0^{\vee} \rangle  =\langle (p-1)\rho-\a_0, \alpha_0^{\vee} \rangle= p(h-2).$ Then  
$$\langle 2(p-1)\rho -a_0, \alpha_0^{\vee} \rangle =\langle \mu_i, \alpha_0^{\vee} \rangle= \langle(p-1)\rho + \la, \alpha_0^{\vee} \rangle, \mbox{ for all } 1 \leq i \leq m-1.$$
Note that the only positive root perpendicular to $\alpha_0$ is $\alpha_2.$ 
Therefore,   all $s_i$ would have to be of the form $s_{\a_2,kp},$ for some  integer $k$ \cite[II.6.1]{rags}. Moreover, for  $s_1$,  we have $k\leq 1.$   
It follows that 
$$s_1 \cdot (2(p-1)\rho - \a_0) = 2(p-1)\rho - \a_0 - (2p-kp-1)\alpha_2 \leq 2(p-1)\rho - \a_0 - (p-1)\alpha_2.$$
One concludes that  $(p-1)\rho + \la$ is equal to $2(p-1)\rho - \a_0 - n \alpha_2$ for some $n \geq p-1.$ 
However,
no weight of this form is contained in $(p-1)\rho + X_+,$ a contradiction.
\end{proof}

We can now show that the TMC holds for $T(2(p-1)\rho - \a_0) $. 

\begin{prop}\label{P:A3p5} Let $G$ be of type  $\rm{A}_3$ with $p=5$.  Then $T(2(p-1)\rho - \a_0)\mid_{G_{1}T} \cong \widehat{Q}_1(\a_0)$ as a $G_1T$-module.
\end{prop}

\begin{proof} Set $ T=T(2(p-1)\rho-\a_0).$ Note that $T$ is projective and injective as a $G_1$-module but not necessarily as a $G$-module in $\textup{Mod}(2(p-1)\rho- \a_0).$ 

It suffices to show that $T$ has a simple $G$-socle, namely $L(\a_0)$. Assume that   there exist weights $\la_0 \in X_1$ and $\la_1 \in X_+$ 
such that 
$(p-1)\rho+w_0\la_0 +p\la_1\neq \a_0$
and  $L((p-1)\rho+w_0\la_0 +p\la_1)$ appears in the socle of $T.$ Then $\widehat{Q}_1((p-1)\rho+ w_0\lambda_0) \otimes L(\la_1)^{(1)}$ has to appear as a $G_1T$-summand of $T$. We set $\la = \la_0 +p\la_1.$ The highest weight of $\widehat{Q}_1((p-1)\rho+ w_0\lambda_0) \otimes L(\la_1)^{(1)}$is $(p-1)\rho +\la.$ This weight is strongly linked to $2(p-1)\rho - \a_0$ and strictly less than $2(p-1)\rho - \a_0$. 

Now Lemma ~\ref{L:2h-32} (with $\s = (p-1)\rho - \a_0$) and Lemma~\ref{L:A3p5}  imply  that 
$T((p-1)\rho+\la_0) \otimes L(\la_1)^{(1)}$ is the injective hull and projective cover of  $L((p-1)\rho+w_0\la_0 +p\la_1)$ in $\textup{Mod}(2(p-1)\rho - \a_0).$ Given the $G$-injections 
$i:L((p-1)\rho+w_0\la_0 +p\la_1) \to T$ and $j:L((p-1)\rho+w_0\la_0 +p\la_1) \to T((p-1)\rho + \la),$ 
there exists a $G$-map $\phi : T \to T((p-1)\rho + \la)$ and a $G_1T$-map $\psi:T((p-1)\rho + \la) \to T$ such that
$i = \phi \circ j$ and $j =  \psi \circ i.$ 
Hence, $i = \phi \circ \psi \circ i.$ Since $T((p-1)\rho + \la)$ has a simple $G$-socle, the $G$-map $\phi$ has to be a surjection. But now the projectivity of $T((p-1)\rho + \la)$ in $\textup{Mod}(2(p-1)\rho- \a_0)$ makes it a $G$-summand, a contradiction. 
\end{proof}

 It follows from Theorem \ref{T:2h-31} that the TMC holds for all $Q_1(\mu)$ as long as $\mu$ is restricted and  not contained in the closure of the lowest alcove. If $\mu$ is in the upper wall of the lowest alcove, the claim follows from Proposition \ref{P:A3p5} by using the translation principle (cf. \cite[II.11.10, II.E.11]{rags}) for other weights in the same facet as $\a_0$. The translation principle also  implies that 
$$T_{\a_0}^{0} [T(2(p-1)\rho - \a_0)] \cong T_{\a_0}^{0} [Q_1(\a_0)] = Q_1(0),$$ as $G_1T$-modules, which forces $T(2(p-1)\rho)\mid_{G_{1}T}\cong Q_1(0).$ Translation within the alcove takes care of the remaining weights in the interior of the lowest alcove. Consequently, the TMC holds for type $\rm{A}_3$ and $p=5$.


\section{Type $\rm{B}_{2}$}

\subsection{} In the case when $\Phi=\rm{B}_{2}$, one has $h = 4$. As the TMC is known to hold for $p \geq 2h - 2$, the only open cases are the primes $p = 2, 3, 5$.  By \cite{KN}, we know that $\St_{1}\otimes L(\mu)$ has a good filtration for all $\mu \in X_1$, and so to verify the TMC we can apply Theorem \ref{T:summary}. Andersen \cite[Example 2 (3)]{And19} verified that $\nabla(\sigma)$ has a good $p$-filtration for all $\sigma\in X_{+}$ 
(with lengthy calculations), and one can appeal to this result to finish off the $\rm{B}_{2}$-case. For the reader's convenience and to make the this paper self-contained, we present short arguments to handle the 
$\Phi=\rm{B}_{2}$, $p = 2, 3, 5$ cases.

\subsection{Type $\rm{B}_{2}$, $p=2$} Since $\langle \rho, \alpha_0^{\vee}\rangle = 3.$ We may make use of Proposition~\ref{P:weights}. The same argument as in the type $\rm{A}_3$ and $p=2$ case works here.

\subsection{Type $\rm{B}_{2}$, $p=3$} Using either the tables in \cite{L} or by applying the Jantzen filtration, one can see that  that any $\eta \in X_+$ with $\Ext_G^1(k, L(\eta))\neq 0$ satisfies $6 \leq \langle \eta, \alpha_0^{\vee} \rangle ,$ with equality holding for $\eta= (1,4).$
Therefore, $p\langle\eta,\a_0^{\vee}\rangle \geq 6p > 6(p-1) = 2(p-1)(h-1)$, and so the claim follows immediately from  Theorem \ref{T:pairingbound}.

\subsection{Type $\rm{B}_{2}$, $p=5$}  We will apply Theorem \ref{T:summary} and show that all $\nabla((p-1)\rho + \mu)$ with $\mu \in X_1$ have a good $p$-filtration. Proposition \ref{P:translation} implies that it is sufficient to show this for all $(p-1)\rho + \mu$ in the principal block. Therefore, we only have to consider $\mu \in \{(1,2), (2,2), (2,4), (4,4)\}.$ We wish to apply Theorem \ref{T:induction}.  

Consider a $G_1B$-composition factor $L(\s_0)\otimes p\s_1$ of $\widehat{Z}_1'((p-1)\rho + \mu)$.  Any weight $\ga$ in the weight lattice of $\St_1$ satisfies $\langle \ga, \alpha_1^{\vee}\rangle \geq -6$ and $\langle \ga, \alpha_2^{\vee}\rangle \geq -12.$ Applying  Proposition \ref{P:weights} to the weights $\mu$ from the above list, one immediately observes that $R^1\ind_{B}^{G} \s_1 \neq 0$ only if   $\alpha_i = \alpha_2$, $\langle \ga, \alpha_2^{\vee}\rangle = -12,$ $\mu = (1,2)$ or $\mu=(2,2),$ and $\sigma_0=0.$ Moreover, we need $\ga+\mu \in pX_+.$ The only pairs $(\ga, \mu)$ satisfying these conditions are $((4,-12), (1,2))$, $((3,-12), (2,2))$ and $((8,-12), (2,2))$. 

In the first two cases $\ga+ \mu = p s_{\alpha_2} \cdot (0,0)$ and in the last case $\ga+ \mu = p s_{\alpha_2} \cdot (1,0).$
As in the type $\rm{A}_3$ with $p=3$ discussion, one observes that the only situation of concern is the existence of a second composition factor $L(0) \otimes p\s$ in $\widehat{Z}_1'((p-1)\rho +\mu)$ such that $\nabla(\s)$  has  $L(0,0)$ or $L(1,0)$ as a composition factor. However, the smallest dominant weight  that is linked to $(0,0)$ is $(2,0)$ and the smallest dominant weight that is linked to $(1,0)$ is $(1,4).$  But, for our choices of $\mu$, neither  $p(2,0)$ nor $p(1,4)$ are less than $(p-1)\rho + \mu.$ Cancellation cannot happen and  $\nabla((p-1)\rho + \mu)$ has the desired $p$-filtration.


\section{Type $\rm{G}_2$, $p\neq 7$}

\subsection{} For $\Phi=\rm{G}_{2}$, $h = 6$. As the TMC is known for $p \geq 2h-2$, this leaves the cases when $p=2$, $3$, $5$ and $7$. 
We will handle the first three cases in this section, and undertake the more detailed $p=7$ case in the next section. 

\subsection{Type $\rm{G}_{2}$, $p=2$} For $p = 2$, in \cite{BNPS2}, the authors have shown that the Tilting Module Conjecture in fact fails in type $\rm{G}_2$.   

\subsection{Type $\rm{G}_2$, $p=3$}\label{S:G2} This case provides a very interesting subexample of the lifting problem. 
Here, there is a strict endomorphism $\sigma: G \to G$ with $\sigma^2 = {F}$ (the Frobenius morphism).  Following the notation in \cite{BNPPSS}, let $G_{1/2}$ denote the scheme-theoretic kernel of $\sigma$. This is a normal subgroup scheme of the first Frobenius kernel $G_1$ of $G$.  We will prove that the projective indecomposable $G_{1/2}$-modules lift to tilting modules for $G$, which will prove that all of those for the Frobenius kernel $G_{1}$ do as well (cf. Proposition~\ref{P:equivTMCr=1}). 

We briefly review some key concepts related to $\sigma$ and $G_{1/2}$. For more details, we refer the reader to \cite[Section 2.2]{BNPPSS} and \cite[Sections 5.3 and 5.4]{Hum}.   Following Bourbaki, $\alpha_1$ will denote the short simple root and $\alpha_2$ will denote the long simple root.  The Lie algebra of $G_{1/2}$ is associated to the ideal of short roots within the Lie algebra of $G$ (a root system of type $\rm{A}_2$).  Given a $G$-module $M$, let $M^{(1/2)}$ denote the module $M$ with action composed with $\sigma$. Note that $(M^{(1/2)})^{(1/2)} \cong M^{(1)}$.   In particular, consider the fundamental simple modules, $L(1,0)^{(1/2)} \cong L(0,1)$ and 
\begin{equation}
L(0,1)^{(1/2)} \cong L(3,0) \cong L(1,0)^{(1)}.
\end{equation}  
More generally, Steinberg's Tensor Product Theorem in this context (for non-negative integers $a$, $b$) is 
\begin{equation}\label{E:TPT}
L(a,b) \cong L(a,0)\otimes L(0,b) \cong L(a,0)\otimes L(b,0)^{(1/2)}.
\end{equation}
Note that the weights $X_{1/2}=\{(a,0) ~|~ 0 \leq a \leq 2\}$ are the $\sigma$-restricted dominant weights, and these index the irreducible $G_{1/2}$-modules, which arise by restriction of $L(a,0)$.  

For a weight $\lambda$, analogous to the definitions of $\widehat{Q}_1(\lambda)$ and $Q_1(\lambda)$ (or $Q_r$ more generally), one can define 
$\widehat{Q}_{1/2}(\lambda)$  to be the injective hull (equivalently projective cover) of $L(\lambda)$ as a $G_{1/2}T$-module
and $Q_{1/2}(\lambda)$ as the injective hull over $G_{1/2}$ (which will be the restriction of $\widehat{Q}_{1/2}(\lambda)$ to $G_{1/2}$).  
As with the ordinary Frobenius kernels, we may identify $G_1/G_{1/2} \cong G_{1/2}^{(1/2)}$, and we have various Lyndon-Hochschild-Serre spectral sequences (e.g., for $G_{1/2} \unlhd G_1$).   The argument as in \cite[Lemma II.11.15]{rags} yields the following result.

\begin{prop}\label{P:Q1tensor} For $0 \leq a, b \leq 2$, there is an isomorphism of $G_{1/2}T$-modules
$$
\widehat{Q}_1(a,b) \cong \widehat{Q}_{1/2}(a,0)\otimes \widehat{Q}_{1/2}(b,0)^{(1/2)}.
$$
\end{prop}
In particular $\widehat{Q}_1(2,2)\cong \St_1 \cong L(2,0) \otimes L(2,0)^{(1/2)} \cong \widehat{Q}_{1/2}(2,0) \otimes \widehat{Q}_{1/2}(2,0)^{(1/2)} .$ We will denote $\widehat{Q}_{1/2}(2,0)$ by $\St_{1/2}.$ It plays the role of the Steinberg module for $G_{1/2}$ and can be regarded as restriction of the $G$-module $L(2,0)$ to $G_{1/2}T.$ The next proposition demonstrates how to construct injective/projective 
modules in the truncated category $\operatorname{Mod}((4,0))$.

\begin{prop}
For $\la \in X_{1/2}$,  $\St_{1/2} \otimes L(\la)$ is injective and projective in $\operatorname{Mod}((4,0))$.
\end{prop}

\begin{proof} 
Let $\eta$ be a dominant weight. By considering the weight multiplicity tables in \cite{L}, one can see that if $\Ext_G^1(k,L(\eta)) \neq 0$, then $\langle \eta,\a_0^{\vee}\rangle \geq 5$.  This arises from the fact that $\eta = (1,1)$ is the smallest weight where there is a non-trivial extension. Clearly, $p \cdot 5=15 > 8 \geq \langle \la + (2,0), \alpha_0^{\vee} \rangle.$
We can now apply the argument in the proof of Proposition \ref{P:injproj}, with $\St_{1/2}$ instead of  $\St_1$ and use part (b) of the same proposition to obtain the claim. 
\end{proof}

Next we adapt the argument in Proposition \ref{P:IfIPTMC}, again with $\St_{1/2}$ in place of $\St_1,$ to conclude:
\begin{prop}
For $\la \in X_{1/2}$,  $T((2,0)+\la) \mid_{G_{1/2}T}\cong \widehat{Q}_{1/2}((2,0)-\la),$ as $G_{1/2}T$-modules.
\end{prop}

Having seen, for $0 \leq a \leq 2$, that each $\widehat{Q}_{1/2}(a,0)$ lifts to $G$ (indeed to $T(4-a,0)$) we obtain the following stronger version of Proposition \ref{P:Q1tensor}, which is an analogue of \cite[Proposition II.11.16(b)]{rags}.

For $0 \leq a, b \leq 2$, and as $G_1T$-modules,
\begin{eqnarray*}
\widehat{Q}_1(a,b) &\cong& \widehat{Q}_{1/2}(a,0)\otimes \widehat{Q}_{1/2}(b,0)^{(1/2)}\\
&\cong& T(4-a,0) \otimes T(4-b, 0)^{(1/2)}.
\end{eqnarray*}
Note that $T(4-a,0) \otimes T(4-b, 0)^{(1/2)}$ is indecomposable as a $G_1T$-module and therefore also as a $G$-module.

The characters of the $L(\la)$ were computed by Springer \cite{Sp}. In particular, $L(\la) \cong \nabla(\la) \cong T(\la)$ for all $\la \in X_{1/2}.$
 Observe that $T(4-a,0) \otimes T(4-b, 0)^{(1/2)}$ is the unique indecomposable  $G$-summand containing the highest weight $(4-a,4-b)$ of 
\begin{align*}
(\St_{1/2} \otimes T(2-a,0)) &\otimes (\St_{1/2} \otimes T(2-b,0))^{(1/2)} \\
&\cong (\St_{1/2} \otimes L(2-a,0)) \otimes (\St_{1/2} \otimes L(2-b,0))^{(1/2)}\\
 &\cong \St_{1/2} \otimes \St_{1/2}^{(1/2)} \otimes L(2-a,0) \otimes L(2-b,0)^{(1/2)}\\
 &\cong \St_1 \otimes L(2-a, 2-b)
 \end{align*}
 which is a tilting module by \cite[8.5.3]{KN}. 
 
We conclude that 
\begin{eqnarray*}
\widehat{Q}_1(a,b) 
&\cong& T(4-a,0) \otimes T(4-b, 0)^{(1/2)} \cong T(4-a,4-b)\mid_{G_{1}T}
\end{eqnarray*}
as $G_1T$-modules. Hence, the TMC holds in this case.

\subsection{Type $\rm{G}_2$, $p = 5$} The verification that the TMC holds follows from Theorem~\ref{T:pairingbound} using the fact that any $\eta \in X^+$ with
$\textup{Ext}_G^1(k,L(\eta)) \ne 0$ 
satisfies $\langle \eta, \alpha_0^{\vee} \rangle \ge 15$ according to \cite{Hag}. This implies that $p\langle \eta, \alpha_0^{\vee} \rangle \geq 75> 40 = 2(p-1)(h-1)$.


\section{Type $\rm{G}_2$, $p=7$}

\subsection{} Since $\St_{1}\otimes L(\mu)$ is known to be tilting for restricted weights $\mu\in X_{1}$  and all primes (\cite{KN} and \cite{BNPS}), we may apply Theorem \ref{T:summary}.  From Proposition \ref{P:translation}, it suffices to show that $\nabla(\hat{\la})$ admits a good $p$-filtration for each $p$-regular weight $\la \in X_1$.  As with previous cases, we use the fact that $\nabla(\hat{\la}) = \ind_{B}^{G}\ \hat{\la} =  \ind_{G_1B}^G\ \widehat{Z}_1'(\hat{\la}) $ and attempt to apply Theorem~\ref{T:induction}.   However, simply knowing the $G_1B$-composition factors of $\widehat{Z}_1'(\hat{\la})$ is not sufficient to obtain the desired conclusion.  We must make a deeper analysis using the $G_1T$-radical filtration of $\widehat{Z}_1'(\hat{\la})$.

\subsection{Computing the $G_1T$-radical filtration}\label{S:explain}  Let $\mu$ be a $p$-regular dominant weight. To compute the $G_1T$-radical filtration of $\widehat{Z}_1'(\mu)$, we make use of a result of Andersen and Kaneda \cite[6.3 Theorem]{AK} (see Theorem \ref{T:AK} below) which relates the radical filtration to inverse Kazhdan-Lusztig polynomials introduced by Lusztig \cite{Lus}.   

To state the theorem, we need some notation involving alcove geometry.  Observe that in this setting each $p$-alcove contains a unique $p$-regular weight, so there is a one-to-one correspondence between alcoves and $p$-regular weights.  Given two alcoves $A, C$, let $Q_{A,C}$ denote the associated inverse Kazhdan-Lusztig polynomial and $d(A,C)$ denote the distance function (an integer determined by the hyperplane reflections needed to reflect from alcove $A$ to alcove $C$).  We will provide one example of these values in Section \ref{S:example} below and refer the reader to \cite{Lus} for details.  A weight $\nu \in X$ is called special if $\nu + \rho \in pX$ (e.g., the base example is $\nu = -\rho$).  Each special point $\nu$ determines a ``box'' in Euclidean space
$$
\Pi_{\nu} := \{ e \in E ~ |  ~ \langle\nu,\a^{\vee}\rangle < \langle e, \a^{\vee}\rangle < \langle\nu,\a^{\vee}\rangle + p \text{ for all } \a \in \Delta\}.
$$
The closures of the boxes (for all possible $\nu$) tessellate the Euclidean space and, in this case, each box contains 12 alcoves.  Let $W_{\nu}$ be the subgroup (isomorphic to the Weyl group) of the affine Weyl group of elements that stabilize $\nu$ under the dot action, and let $w_{\nu}$ denote the longest word in $W_{\nu}$.  For example, when $\nu = -\rho$, $w_{\nu} = w_0$.  For $\la \in \Pi_{\nu}$, $w_{\nu}\cdot \la = -\la + 2\nu$.

Given a finite-dimensional $G_1T$-module $M$, consider the radical filtration
$$
0 = \rad^n M  \subseteq \rad^{n-1} M \subseteq \cdots \subseteq \rad^1 M  \subseteq \rad^0 M = M.
$$
For $0 \leq j \leq n$, set $\rad_j M := \rad^j M /\rad^{j+1} M $, the $j$th radical layer.

We now restate Theorem 6.3 of \cite{AK} with our conventions:

\begin{theorem}[Andersen-Kaneda]\label{T:AK} Let $A$ be an alcove with $p$-regular weight $\mu$, $\nu$ be a special point, and $C$ be an alcove in $\Pi_{\nu}$ with $p$-regular weight $\la$.  Then
$$
Q_{A,C} = \sum_{j}q^{\frac12(d(A,C) - j)}[\rad_j\widehat{Z}'_1(\mu):\widehat{L}_1(w_{\nu}\cdot\la)].
$$
\end{theorem}

\begin{remark} The theorem holds for $p > h$ under the assumption that the $G_1T$-version of the Lusztig Conjecture holds.  See the discussion in Sections 5.1 and 5.2 of \cite{AK} as well as the discussion in \cite[II.D.13]{rags}.  In \cite{rags}, Jantzen works under an assumption $(\widehat{D})$ on the semi-simplicity of certain $G_1T$-modules (cf. \cite[II.C.14, II.D.4]{rags}). We refer the reader to \cite[II.C.17]{rags} for equivalent conditions. Note that for all rank 2 groups the modules $\nabla(\la)$ with  $p$-regular highest weight $\la \in X_1$ have multiplicity free composition series. This immediately implies $(\widehat{D})$. In short, the theorem is known to be valid for rank 2 groups and $p >h.$
\end{remark}

\subsection{An Example: $\widehat{Z}_1'(0,0)$}\label{S:example}
Consider the module $\widehat{Z}_1'(0,0)$.   We give a partial demonstration of how to apply Theorem \ref{T:AK} to obtain the $G_1T$-radical layers.  Let alcove $A$ be the alcove containing (0,0).  Take as the special point $\nu = -\rho$.  We may apply the theorem to each alcove $C$ in $\Pi_{-\rho}$.  Using the values of $Q_{A,C}$ from \cite{Lus} and computing the values of $d(A,C)$ by hand, one can identify the allowable values of $j$ and the corresponding composition multiplicity. The results are summarized in the following table, where the alcoves are numbered following \cite[18.2 Figure 1]{Hum}, $\la$ is the $p$-regular weight in $C$, and $w_{-\rho}\cdot \la = \s_0 - 7\rho = \s_0 + 7(-\rho)$ for $\s_0 \in X_1$ (noting that $\widehat{L}_1(w_{-\rho}\cdot\la) \cong \widehat{L}_1(\s_0)\otimes 7(-\rho) \cong L(\s_0)\otimes 7(-\rho)$):

\medskip
\begin{tabular}{|c|c|c|c|c|c|c|c|}
\hline
$C$ & $\la$ & $w_{-\rho}\cdot \la$ & $\s_0$ & $Q_{A,C}$ & $d(A,C)$ & $j$ & $[\rad_j\widehat{Z}_1'(0,0) : \widehat{L}_1(w_{-\rho}\cdot\la)]$ \\
\hline
1 & (0,0) & (-2,-2) & (5,5) & 1 & 0 & 0 & 1\\
2 & (2,0) & (-4,-2) & (3,5) & 1 & 1 & 1 & 1\\
3 & (1,1) & (-3,-3) & (4,4) & 1 & 2 & 2 & 1\\
4 & (1,2) & (-3,-4) & (4,3) & 1 & 3 & 3 & 1\\
5 & (2,2) & (-4,-4) & (3,3) &  $q$ & 4 & 2 & 1\\
6 & (0,4) & (-2,-6) & (5,1) &  $q$ & 5 & 3 & 1\\
7 & (5,1) & (-7,-3) & (0,4) &  $q^2$ & 5 & 1 & 1\\
8 & (3,3) & (-5,-5) & (2,2) & $q + q^2$ & 6 & 2 & 1\\
	& & & & & & 4 & 1\\
11 & (4,3) & (-6,-5) & (1,2) & $2q^2 + q^3$ & 7 & 1 & 1\\
	& & & & & & 3 & 2\\
13 & (4,4) & (-6,-6) & (1,1) &  $3q^3$ & 8 & 2 & 3\\
15 & (3,5) & (-5,-7) & (2,0) & $3q^4$ & 9 & 1 & 3\\
16 & (5,5) & (-7,-7) & (0,0) & $q^3 + 3q^4$ & 10 & 2 & 3\\
	& & & & & & 4 & 1\\
\hline
\end{tabular}

\medskip
In a similar manner, we consider all possible special points $\nu$ for which there is an alcove $C$ lying in the box $\Pi_{\nu}$ with $Q_{A,C} \neq 0$ and make computations as above.   In total there are 16 relevant special points (including $-\rho$).  We may write $\nu = -\rho + 7\tilde{\nu}$ for some weight $\tilde{\nu}$.  Then, if $\la'$ is a $p$-regular weight in $\Pi_{\nu}$, $\la' = \la + 7\tilde{\nu}$ for $\la \in \Pi_{-\rho}$ as above.  Furthermore, 
$$
w_{\nu}\cdot \la' = -\la' + 2\nu = -\la - 2\rho + 7\tilde{\nu} = w_{-\rho}\cdot\la + 7\tilde{\nu} = \s_0 + 7(-\rho + \tilde{\nu}),
$$
with $\s_0 \in X_1$ as above.  So $\widehat{L}_1(w_{\nu}\cdot\la') \cong \widehat{L}_1(\s_0)\otimes 7(-\rho + \tilde{\nu}) \cong L(\s_0)\otimes 7\s_1$, by setting $\s_1 := -\rho + \tilde{\nu}$.  In particular, the above table contains all $\s_0$ that appear, that is, gives all the isotypic components $L(\s_0)$ that will appear within $\widehat{Z}_1'(0,0)$.   Moreover, the $\s_1$-portion is easily obtained from $\tilde{\nu}$.   Putting all of these cases together, we get the results in the table labelled ``Alcove 1'' in Appendix \ref{S:appendix}. 

From the general fact that $\widehat{Z}_1'(\mu_0 + p\mu_1) \cong \widehat{Z}_1'(\mu_0)\otimes p\mu_1$ as a $G_1B$-module, the filtration for $\widehat{Z}_1'(0,0)$ also gives us the filtration for $\widehat{Z}_1'(7,7) = \widehat{Z}_1'(0,0)\otimes 7(1,1)$, which is one of the modules of interest for us.  The aforementioned table lists the values of $\s_1$ for this module as well, which are obtained by simply adding $(1,1)$ to the previous ones.  

Tables for the other 11 alcoves are also given in  Appendix \ref{S:appendix}.

\subsection{Composition Factor Analysis} Let $\mu = \hat{\la}$ with $\la \in X_1$ being $p$-regular and consider Theorem~\ref{T:induction}.  For $\nabla(\mu)$ to fail to have a good $p$-filtration, we must have composition factors  $L(\s_0)\otimes p\s_1$ and $L(\s_0)\otimes p\tilde{\s}_1$ (same isotypic component)
with a non-zero and non-isomorphic $G$-map $\nabla(\s_1) \to  R^1\ind_B^G\tilde{\s}_1$.    
Next, observe that if $\nabla(\s_1) \cong L(\s_1)$ for all such $\s_1$, then any non-zero map would necessarily be an isomorphism.  Reviewing the table of composition factors in Appendix \ref{S:appendix}, we see that there are only two cases where 
$\nabla(\s_1) \neq L(\s_1)$: $\s_1 = (2,0), (1,1)$.   Precisely, we have two short exact sequences of $G$-modules:
$$
0 \to L(2,0) \to \nabla(2,0) \to k \to 0
$$
and
$$
0 \to L(1,1) \to \nabla(1,1) \to L(2,0) \to 0.
$$

\subsection{Composition Factors: Non-vanishing $R^1$}  Considering again the collection of all $\s_1$ that appear in the tables of Appendix \ref{S:appendix}, we are interested in those $\s_1$ where $R^1\ind_{B}^G\s_1 \neq 0$.  For this to occur, we must have (cf. \cite[Prop. II.5.4]{rags}) $\langle\s_1,\a^{\vee}\rangle \leq -2$ for a simple root $\a$.  The set of such weights that appear are as follows: 
$$(-2,1), (-2,2), (-2,3), (3,-2), (-3,1), (-3,2), (-3,3), (-4,2), (-4,3), (4,-2), (5,-2).$$ 
Using \cite[II.5.4(d)]{rags}, one can determine the structure of $R^1\ind_{B}^G\s_1$ for each such $\s_1$.  Note that the simple roots are $\a_1 = (2,-1)$ and $\a_2 = (-3,2)$.  

\begin{itemize}
\item $R^1\ind_B^G(-2,1) = R^1\ind_B^G (s_{\a_1}\cdot (0,0))  \cong \nabla(0,0) = k$
\item $R^1\ind_B^G(-2,2)  = R^1\ind_B^G(s_{\a_1}\cdot (0,1))\cong \nabla(0,1) = L(0,1)$
\item $R^1\ind_B^G(-2,3) = R^1\ind_B^G(s_{\a_1}\cdot (0,2)) \cong\nabla(0,2) = L(0,2)$
\item $R^1\ind_B^G(3,-2) = R^1\ind_B^G(s_{\a_2}\cdot (0,0)) \cong \nabla(0,0) = k$
\item $R^1\ind_B^G(-3,1) = R^1\ind_B^G(s_{\a_1}\cdot (1,-1)) \cong \nabla(1,-1) = 0$
\item $R^1\ind_B^G(-3,2) = R^1\ind_B^G(s_{\a_1}\cdot (1,0))\cong \nabla(1,0) = L(1,0)$
\item $R^1\ind_B^G(-3,3) = R^1\ind_B^G(s_{\a_1}\cdot (1,1)) \cong \nabla(1,1)$, with socle $L(1,1)$
\item $R^1\ind_B^G(-4,2) = R^1\ind_B^G(s_{\a_1}\cdot (2,-1)) \cong \nabla(2,-1) = 0$
\item $R^1\ind_B^G(-4,3) = R^1\ind_B^G(s_{\a_1}\cdot (2,0)) \cong \nabla(2,0)$, with socle $L(2,0)$
\item $R^1\ind_B^G(4,-2) = R^1\ind_B^G(s_{\a_2}\cdot (1,0))\cong \nabla(1,0) = L(1,0)$
\item $R^1\ind_B^G(5,-2) = R^1\ind_B^G(s_{\a_2}\cdot (2,0)) \cong \nabla(2,0)$, with socle $L(2,0)$
\end{itemize}

From the discussion above, one can also see from \cite[Prop. II.5.4]{rags} that $R^2\ind_B^G \s_1 = 0$ for all $\s_1$ that appear; a condition needed to apply part (c)(iii) of Theorem \ref{T:induction}.  

\subsection{Composition Factors: Comparison} Using the previous two sections, we see that there are only four potentially bad scenarios, where the following pairs ($\s_1$, $\tilde{\s}_1$) arise (with a common $\s_0$):

{\bf Case 1:} $\s_1 = (2,0)$ and $\tilde{\s}_1 = (3,-2)$, with $\s_0 = (0,0)$

{\bf Case 2:} $\s_1 = (2,0)$ and $\tilde{\s}_1 = (-2,1)$, with various $\s_0$: $(0,0)$, $(2,0)$, $(1,1)$, $(1,2)$, $(2,2)$

{\bf Case 3:} $\s_1 = (1,1)$ and $\tilde{\s}_1 = (-4,3)$, with $\s_0 = (0,0)$

{\bf Case 4:} $\s_1 = (1,1)$ and $\tilde{\s}_1 = (5,-2)$, with $\s_0 = (0,0)$

\medskip\noindent
For $\la \in X_1$, the numbering of the alcoves in Appendix \ref{S:appendix} corresponds to that of $\hat{\la}$ as in the following table.

\medskip
\begin{tabular}{|c|c|c|c|c|c|c|}
\hline
Alcove & 1 & 2 & 3 & 4 & 5 & 6\\
\hline
$\la$ & (5,5) & (3,5) & (4,4) & (4,3) & (3,3) & (4,1) \\
\hline
$\hat{\la}$ &  (7,7) & (9,7) & (8,8) & (8,9) & (9,9) & (7,11)  \\
\hline \hline 
Alcove &  7 & 8 & 11 & 13 & 15 & 16\\
\hline
$\la$ & (0,4) & (2,2) & (1,2) & (1,1) & (2,0) & (0,0)\\
\hline
$\hat{\la}$ &  (12,8) & (10,10) & (11,10) & (11,11) & (10,12) & (12,12) \\
\hline
\end{tabular}

\medskip\noindent
Considering all the composition factors in the tables of Appendix \ref{S:appendix}, we find that Case 1 arises in Alcoves 1 and 2.  Case 2 arises in Alcoves 1 through 8, but not 11 through 16.    Case 3 occurs only in Alcove 4. Lastly, Case 4 occurs only in Alcove 11.  Therefore, $\nabla(\hat{\lambda})$ has a good $p$-filtration for Alcoves 13, 15, and 16. 
In the Case 1 scenarios, when the weight $(3,-2)$ appears in the same isotypic component as the weight $(2,0)$, it always appears in a strictly higher level $G_1T$-radical layer.   Hence, that would also be the case for the $G_1B$-radical layers.   As such no cancellation can occur (cf. condition (b) of \ref{T:induction}), and this case is resolved.
In the Case 2 scenarios, the weight (-2,1) appears in the same or higher level $G_1T$-radical layer.   In particular, for Alcoves 5 and 6, the weights only occur in different $G_1T$-radical layers (with the weight (-2,1) above the weight (2,0)).  As above, since this relationship would also hold for the $G_1B$-radical filtration, then we would see that no cancellation will occur.  In the remaining alcoves, we have only to consider those cases where the weights (-2,1) and (2,0) lie in the same $G_1T$-radical layer.    We note that in occurrences of Case 3 and Case 4, the weights of concern also arise in the same $G_1T$-radical layer.  

\subsection{Case 2} Consider the case when $\widehat{Z}_1'(\mu)$ contains composition factors of the form $L(\s_0)\otimes 7(2,0)$ and $L(\s_0)\otimes 7(-2,1)$ for some $\s_0$.  As noted above, this is of concern only if they lie in the same $G_1T$-radical layer.   For cancellation to occur, there must be a non-trivial extension $\Ext_B^i((2,0),(-2,1))$ and an associated $B$-module $N$ (i.e., with (2,0) in the head and (-2,1) in the socle) such that $L(\s_0)\otimes N^{(1)}$ appears as a $G_1B$-subquotient of $\widehat{Z}_1'(\mu)$. Since (length one) extensions of $B$-modules differ by $p^n\a$ for a simple root $\a$, within $N$, the weight (module) lying immediately above $(-2,1) = -\a_1$ must be $-\a_1 + p^n\a$ for $\a \in \{\a_1,\a_2\}$.  A check of all composition factors shows that $n = 0$ and that only $-\a_1 + \a_1 = 0$ occurs.  In particular, $-\a_1 + \a_2 = (-5,3)$ does not occur.  That is, we must be in a situation where we have a non-split (or cancellation could not occur) short exact sequence of $B$-modules:
$$
0 \to -\a_1 \to S \to k \to 0
$$
and 
$$
0 \to S \to N \to N/S \to 0.
$$
Consider the long exact sequence in induction associated to the first sequence:
$$
0 \to \ind_B^G(-\a_1) \to \ind_B^G S \to \ind_B^G k \to R^1\ind_{B}^G(-\a_1) \to R^1\ind_B^G S  \to R^1\ind_B^G k  \to \cdots.
$$
By basic facts as above, this becomes
$$
0 \to 0 \to \ind_B^G S  \to k \overset{\phi}{\to} k \to R^1\ind_B^G S  \to 0 \to \cdots.
$$
Moreover, the map $\phi$ is necessarily non-zero (and hence an isomorphism) since the original extension was non-split. We are using here the identification 
$\Hom_G(k,R^1\ind_B^G(-\a_1)) \cong \Ext^1_{B}(k,-\a_1)$.   Exactness now implies that $\ind_B^G S  = 0$ and $R^1\ind_B^G S = 0$.

Consider now the long exact sequence associated to the second short exact sequence above:
$$
0 \to \ind_B^G S  \to \ind_B^G N  \to \ind_B^G N/S  \to R^1\ind_B^G S  \cdots.
$$
Using the preceding conclusions, we see that $\ind_B^G N  \cong \ind_B^G N/S $.  Repeating this argument as needed, any weights of the form $(-2,1)$ can be removed from consideration and will not lead to any cancellation.  Thus, $\nabla(\hat{\la})$ admits a good $p$-filtration for all alcoves except possibly 4 and 11.

\subsection{Case 3} Consider the case when $\widehat{Z}_1'(\mu)$ contains composition factors of the form $L(\s_0)\otimes 7(1,1)$ and $L(\s_0)\otimes 7(-4,3)$ for some $\s_0$.  As noted above, this occurs only in Alcove 4 within a common $G_1T$-radical layer (and with $\s_0 = 0$).   For cancellation to occur, there must be a non-trivial extension $\Ext_B^i((1,1),(-4,3))$ and an associated $B$-module $N$ (i.e., with (1,1) in the head and (-4,3) in the socle) such that $L(\s_0)\otimes N^{(1)}$ appears as a $G_1B$-subquotient of $\widehat{Z}_1'(\mu)$. As above, within $N$, the weight (module) lying immediately above $(-4,3)$ is either $(-4,3) + \a_1 = (-2,2)$ or $(-4,3) + \a_2 = (-7,5)$. A check of all composition factors shows that the latter case does not occur. Continuing in this manner, above the weight (-2,2) must lie either (0,1) or (-5,4), but again the latter does not occur, and above (0,1) must lie either (2,0) or (-3,3), but once again, the latter does not occur.  So we may assume that we have a series of non-split short exact sequences of $B$-modules:
$$
0 \to S_1 \to N \to N/S_1 \to 0,
$$
$$
0 \to S_2 \to S_1 \to  (2,0) \to 0,
$$
$$
0 \to (-4,3) \to S_2 \to S_3 \to 0,
$$
and
$$
0 \to (-2,2) \to S_3 \to (0,1) \to 0.
$$
Arguing as before with the associated long exact sequences, starting from the bottom, we find that $\ind_B^G S_3 = 0 = R^1\ind_B^G S_3 $.  Then we get $\ind_B^G S_2 = 0$ and $R^1\ind_B^G S_2  \cong R^1\ind_{B}^{G} (-4,3) \cong \nabla(2,0)$.  This gives $\ind_B^G S_1  = 0 = R^1\ind_B^G S_1 $, from which we get as before that $\ind_B^G N \cong \ind_B^G N/S_1$. Hence, the weight $(-4,3)$ is removed and cannot cancel out the head of $\nabla(1,1)$.

\subsection{Case 4}  Consider the case when $\widehat{Z}_1'(\mu)$ contains composition factors of the form $L(\s_0)\otimes 7(1,1)$ and $L(\s_0)\otimes 7(5,-2)$ for some $\s_0$. As noted above, this occurs only in Alcove 11 within a common $G_1T$-radical layer (and again with $\s_0 = 0$). The situation may be resolved in a manner similar to Case 2.  As in that case, suppose we have a non-trivial extension $\Ext^i_B((1,1),(5,-2))$ with associated module $N$.  Within $N$, lying directly above $(5,-2)$ must be $(5,-2) + \a_1 = (7,-3)$ or $(5,-2) + \a_2 = (2,0)$.  The former case does not occur. So $N$ contains a submodule $S$ with structure given by a non-split extension
$$
0 \to (5,-2) \to S \to (2,0) \to 0.
$$
Consider again the associated long exact sequence:
\begin{align*}
0 \to \ind_B^G(5,-2) &\to \ind_B^G S \to \ind_B^G(2,0) \to R^1\ind_{B}^G(5,-2) \to R^1\ind_B^G S  \to \\
	&\to R^1\ind_B^G(2,0) \to \cdots,
\end{align*}
which reduces to
$$
0 \to 0 \to \ind_B^G S  \to \nabla(2,0) \overset{\phi}{\to} \nabla(2,0) \to R^1\ind_B^G S  \to 0 \to \cdots,
$$
where once again the map $\phi$ must be an isomorphism (the only option for a non-trivial map based on the structure of $\nabla(2,0)$).  And so the argument proceeds as above.

\subsection{Additional Remarks} In the Appendix, as an added bonus, we provide complete information about the $G_{1}T$-radical series for the baby Verma modules for $\Phi=\rm{G}_{2}$ when $p=7$. 
We should also add that our proof in this section shows that if one considers a layer of the $G_{1}B$-radical series for the baby Verma modules and applies $\text{ind}_{G_{1}B}^{G}(-)$ then the resulting $G$-module has a good $p$-filtration. 

One can also use the data in those tables to compute the Ext-groups $\Ext_{G_1}^1(L(\la),L(\mu))$ for $p$-regular dominant weights $\la, \mu$ in these alcoves, verifying the results in \cite[Section 4.2, Figure 3]{Lin} for all primes greater than or equal to $7$.

\newpage
\appendix
\section{$G_1T$-radical series of baby Verma modules}\label{S:appendix}

The following tables give the $G_1T$-composition factors (necessarily also $G_1B$-composition factors) of the radical layers of $\widehat{Z}_1'(\mu)$ and $\widehat{Z}_1'(\mu + 7\rho) \cong \widehat{Z}_1'(\mu)\otimes 7\rho$ for each 7-regular $\mu \in X_1$.  Here, a composition factor is $L(\s_0)\otimes 7\s_1$, with $j$ denoting the layer $\rad_j\widehat{Z}_1'(\mu) := \rad^j\widehat{Z}_1'(\mu)/\rad^{j+1}\widehat{Z}_1'(\mu)$. We refer the reader to Section~\ref{S:example} for an outline for the calculations for Alcove 1. The remaining alcoves (cf. 
\cite[18.2 Figure 1]{Hum})  in the restricted region are treated analogously. 
\\\\\\
\begin{center} {\bf Alcove 1}

\medskip\noindent
\begin{tabular}{|c|c|c|c|}
\hline
$\s_0$ & $j$ & $\s_1$ for $\widehat{Z}_1'(0,0)$ & $\s_1$ for $\widehat{Z}_1'(7,7)$\\
\hline
(5,5) & 0 & (-1,-1) & (0,0)\\
\hline
(3,5) & 1 & (-1,-1) & (0,0)\\
\hline
	& 5 & (0,-1) & (1,0)\\
\hline
(4,4) & 2 & (-1,-1) & (0,0)\\
\hline
	& 4 & (0,-1), (-2,0) & (1,0), (-1,1)\\
\hline
(4,3) & 3 & (-1,-1),  (0,-1), (1,-2), (-2,0) & (0,0), (1,0), (2,-1), (-1,1)\\
\hline
(3,3) & 2 & (-1,-1),(0,-1), (1,-2),  (-2,0), (-3,0) & (0,0), (1,0), (2,-1), (-1,1), (-2,1)\\
\hline
(5,1) & 3 & (-1,-1),(0,-1),  (1,-2), (-2,0), (-3,0) & (0,0), (1,0), (2,-1), (-1,1), (-2,1)\\
\hline
	& 5 & (-1,0) & (0,1)\\
\hline
(0,4) & 1 & (-1,-1),  (0,-1),  (0,-2), (1,-2), (-2,0), & (0,0),  (1,0), (1,-1), (2,-1), (-1,1),\\
	& & (-3,0) & (-2,1)\\
\hline
(2,2) & 2 & (-1,-1),  (0,-1),  (0,-2), (1,-2), (-2,0), & (0,0),  (1,0), (1,-1), (2,-1), (-1,1),\\
	& & (-3,0) & (-2,1)\\
\hline
	& 4 & (-1,-1),  (0,-1), (-1,0), (1,-2), (-2,0), & (0,0),  (1,0), (0,1), (2,-1), (-1,1),\\
	& & (2,-2) & (3,-1)\\
\hline
(1,2) & 1  & (-1,-1),  (0,-1), (0,-2), (1,-2), (-2,0),   & (0,0), (1,0), (1,-1),  (2,-1),  (-1,1), \\
	&	 &  (-2,-1), (-3,0) &(-1,0),  (-2,1)\\
\hline
	& 3 & (-1,-1)$^{\oplus 2}$, (0,-1)$^{\oplus 2}$, (-1,0),  (1,-2), &  (0,0)$^{\oplus 2}$, (1,0)$^{\oplus 2}$, (0,1), (2,-1),\\
	& 	& (-2,0)$^{\oplus 2}$,  (2,-2), (-3,0), (-4,1) & (-1,1)$^{\oplus 2}$, (3,-1), (-2,1), (-3,2)\\
\hline
(1,1) & 2 & (-1,-1)$^{\oplus 3}$, (0,-1)$^{\oplus 2}$, (-1,0), (0,-2),  & (0,0)$^{\oplus 3}$, (1,0)$^{\oplus 2}$, (0,1), (1,-1),\\
	& 	&  (1,-2)$^{\oplus 2}$,  (-2,0)$^{\oplus 2}$, (-2,-1), (2,-2),  &  (2,-1)$^{\oplus 2}$, (-1,1)$^{\oplus 2}$, (-1,0), (3,-1),  \\
	&	&  (-3,0), (2,-3), (-4,1) & (-2,1), (3,-2), (-3,2)\\
\hline
	& 4 & (0,-1) & (1,0)\\
\hline
(2,0) & 1 & (-1,-1)$^{\oplus 3}$,  (0,-1), (-1,0), (0,-2), & (0,0)$^{\oplus 3}$,  (1,0), (0,1), (1,-1),\\
	& 	&  (1,-2), (-2,0), (-2,-1), (2,-2),  & (2,-1), (-1,1), (-1,0), (3,-1)\\
	&	&  (-3,0), (2,-3), (-4,0),  (-4,1) & (-2,1), (3,-2), (-3,1), (-3,2)\\
\hline
	& 3 & (0,-1), (-2,0) & (1,0), (-1,1)\\
\hline
(0,0) & 2 & (-1,-1)$^{\oplus 3}$,  (0,-1), (-1,0), (0,-2), & (0,0)$^{\oplus 3}$, (1,0), (0,1), (1,-1), \\
	& 	&   (1,-2), (-2,0),  (-2,-1), (2,-2), & (2,-1),  (-1,1), (-1,0), (3,-1),  \\
	&	&   (-3,0), (2,-3), (-4,0),  (-4,1) & (-2,1), (3,-2), (-3,1), (-3,2)\\
\hline
	& 4 & (-1,-1), (0,-1)$^{\oplus 3}$, (-1,0), (1,-1),   & (0,0), (1,0)$^{\oplus 3}$, (0,1), (2,0),  \\
	& 	&  (0,-2), (1,-2), (-2,0)$^{\oplus 2}$, (2,-2),   &  (1,-1), (2,-1), (-1,1)$^{\oplus 2}$, (3,-1), \\
	&      &  (-3,0), (-3,1), (3,-3) & (-2,1),  (-2,2), (4,-2)\\
\hline
	& 6 & (0,0) & (1,1)\\
\hline
\end{tabular}
\end{center}

\newpage
\begin{center} {\bf Alcove 2}

\medskip\noindent
\begin{tabular}{|c|c|c|c|}
\hline
$\s_0$ & $j$ & $\s_1$ for $\widehat{Z}_1'(2,0)$ & $\s_1$ for $\widehat{Z}_1'(9,7)$\\
\hline
(5,5) & 5 & (0,-1) & (1,0)\\
\hline
(3,5) & 0 & (-1,-1) & (0,0)\\
\hline
	& 4 & (0,-1) & (1,0)\\
\hline
(4,4) & 1 & (-1,-1) & (0,0)\\
\hline
	& 3 & (0,-1), (-2,0) & (1,0), (-1,1)\\
\hline
(4,3) & 2 & (-1,-1),  (0,-1), (1,-2), (-2,0) & (0,0), (1,0), (2,-1), (-1,1)\\
\hline
(3,3) & 3 & (-1,-1), (0,-1), (1,-2),  (-2,0) & (0,0), (1,0), (2,-1), (-1,1)\\
\hline
	& 5 & (-1,0) & (0,1)\\
\hline
(5,1) & 2 & (-1,-1), (0,-1),  (1,-2), (-2,0), (-3,0) & (0,0), (1,0), (2,-1), (-1,1), (-2,1)\\
\hline
	& 4 & (-1,0) & (0,1)\\
\hline
(0,4) & 4 & (-1,-1),  (0,-1),  (-1,0), (1,-2), (-2,0), & (0,0),  (1,0), (0,1), (2,-1), (-1,1),\\
	&  & (2,-2) & (3,-1)\\
\hline
(2,2) & 1 & (-1,-1),  (0,-1),  (0,-2), (1,-2), (-2,0), & (0,0),  (1,0), (1,-1), (2,-1), (-1,1),\\
	& & (-3,0) & (-2,1)\\
\hline
	& 3 & (-1,-1),  (0,-1), (-1,0), (1,-2), (-2,0), & (0,0),  (1,0), (0,1), (2,-1), (-1,1),\\
	& & (2,-2) & (3,-1)\\
\hline
(1,2) & 2  & (-1,-1)$^{\oplus 2}$,  (0,-1), (0,-2), (1,-2),   & (0,0)$^{\oplus 2}$, (1,0), (1,-1),  (2,-1), \\
	&	 &  (-2,0),  (-3,0) &   (-1,1),  (-2,1)\\
\hline
	& 4 & (-1,-1), (0,-1)$^{\oplus 2}$, (-1,0)$^{\oplus 2}$, (1,-1), &  (0,0), (1,0)$^{\oplus 2}$, (0,1)$^{\oplus 2}$, (2,0), \\
	& 	&  (1,-2), (-2,0)$^{\oplus 2}$,  (2,-2), (-3,1)& (2,-1), (-1,1)$^{\oplus 2}$, (3,-1), (-2,2)\\
\hline
(1,1) & 1 & (-1,-1) & (0,0)\\
\hline
	& 3 & (-1,-1)$^{\oplus 2}$, (0,-1)$^{\oplus 3}$, (-1,0), (1,-1),  & (0,0)$^{\oplus 2}$, (1,0)$^{\oplus 3}$, (0,1), (2,0)\\
	& 	&   (0,-2), (1,-2)$^{\oplus 2}$,  (-2,0)$^{\oplus 2}$,  (2,-2),  &  (1,-1),  (2,-1)$^{\oplus 2}$, (-1,1)$^{\oplus 2}$,  (3,-1),  \\
	&	&  (-3,0),  (-3,1), (3,-3)  & (-2,1), (-2,2), (4,-2)\\	
\hline
(2,0) & 2 & (-1,-1) & (0,0)\\
\hline
	& 4  & (-1,-1),  (0,-1)$\ppp$, (-1,0), (1,-1),  & (0,0),  (1,0)$\ppp$, (0,1), (2,0),\\
	& 	&  (0,-2), (1,-2), (-2,0)$\pp$, (2,-2),  &  (1,-1), (2,-1), (-1,1)$\pp$, (3,-1)\\
	&	&  (-3,0), (-3,1), (3,-3) & (-2,1), (-2,2), (4,-2)\\
\hline
	& 6 & (0,0) & (1,1)\\
\hline
(0,0) & 1 & (-1,-1)$^{\oplus 3}$,  (0,-1), (-1,0), (0,-2), & (0,0)$^{\oplus 3}$, (1,0), (0,1), (1,-1), \\
	& 	&   (1,-2), (-2,0), (-2,-1), (2,-2), & (2,-1),  (-1,1), (-1,0),  (3,-1), \\
	&	&   (-3,0), (2,-3), (-4,0),  (-4,1) & (-2,1), (3,-2), (-3,1), (-3,2)\\
\hline
	& 3 & (-1,-1), (0,-1)$^{\oplus 3}$, (-1,0), (1,-1),   & (0,0), (1,0)$^{\oplus 3}$, (0,1), (2,0),  \\
	& 	&  (0,-2), (1,-2), (-2,0)$^{\oplus 2}$, (2,-2),   &  (1,-1), (2,-1), (-1,1)$^{\oplus 2}$, (3,-1), \\
	&      &  (-3,0), (-3,1), (3,-3) & (-2,1), (-2,2),  (4,-2)\\
\hline
	& 5 & (0,0) & (1,1)\\
\hline
\end{tabular}
\end{center}

\newpage
\begin{center} {\bf Alcove 3}

\medskip\noindent
\begin{tabular}{|c|c|c|c|}
\hline
$\s_0$ & $j$ & $\s_1$ for $\widehat{Z}_1'(1,1)$ & $\s_1$ for $\widehat{Z}_1'(8,8)$\\
\hline
(5,5) & 4 & (0,-1) & (1,0)\\
\hline
(3,5) & 3 &  (0,-1), (-2,0) & (1,0), (-1,1)\\
\hline
(4,4) & 0 & (-1,-1) & (0,0)\\
\hline
	& 2 & (0,-1), (-2,0) & (1,0), (-1,1)\\
\hline
(4,3) & 1 & (-1,-1) & (0,0)\\
\hline
	& 3 & (0,-1), (-2,0) & (1,0), (-1,1)\\
\hline
	& 5 & (-1,0) & (0,1)\\
\hline
(3,3) & 2 & (-1,-1), (0,-1), (1,-2),  (-2,0) & (0,0), (1,0), (2,-1), (-1,1)\\
\hline
	& 4 & (-1,0) &  (0,1)\\
\hline
(5,1) & 1 & (-1,-1), (0,-1),  (1,-2), (-2,0), (-3,0) & (0,0), (1,0), (2,-1), (-1,1), (-2,1)\\
\hline
	& 3 & (-1,0) & (0,1)\\
\hline
(0,4) & 3 & (-1,-1),  (0,-1),  (1,-2), (-2,0), & (0,0),  (1,0), (2,-1), (-1,1),\\
\hline
	& 5 & (1,-1), (-1,0) & (2,0), (0,1)\\
\hline
(2,2) & 2 & (-1,-1)$\pp$,  (0,-1),  (1,-2), (-2,0), & (0,0)$\pp$,  (1,0),  (2,-1), (-1,1),\\
	& & (-3,0) & (-2,1)\\
\hline
	& 4 &  (0,-1), (-1,0)$\pp$, (1,-1), (-2,0), &  (1,0), (0,1)$\pp$, (2,0), (-1,1),\\
	& & (-3,1) & (-2,2)\\
\hline
(1,2) & 1  & (-1,-1)$^{\oplus 2}$,  (0,-1), (0,-2), (1,-2),   & (0,0)$^{\oplus 2}$, (1,0), (1,-1),  (2,-1), \\
	&	 &  (-2,0),  (-3,0) &   (-1,1),  (-2,1)\\
\hline
	& 3 & (-1,-1), (0,-1)$^{\oplus 2}$, (-1,0)$^{\oplus 2}$, (1,-1), (1,-2), &  (0,0), (1,0)$^{\oplus 2}$, (0,1)$^{\oplus 2}$, (2,0), (2,-1),\\
	& 	& (-2,0)$^{\oplus 2}$,  (2,-2), (-3,1)& (-1,1)$^{\oplus 2}$, (3,-1), (-2,2)\\
\hline
(1,1) & 2 & (-1,-1)$\pp$, (0,-1), (0,-2), (1,-2) & (0,0)$\pp$, (1,0), (1,-1), (2,-1)\\
	&  & (-2,0), (-3,0) & (-1,1), (-2,1)\\
\hline
	& 4 & (0,-1)$\pp$, (1,-1), (-1,0)$\pp$, (-2,0)$\pp$, & (1,0)$\pp$, (2,0), (0,1)$\pp$, (-1,1)$\pp$,\\
	&  & (2,-2), (-3,1) & (3,-1), (-2,2)\\
\hline
	& 6 & (0,0) & (1,1)\\
\hline
(2,0) & 1 & (-1,-1) & (0,0)\\
\hline
	& 3  & (-1,-1),  (0,-1)$\ppp$, (-1,0), (1,-1),  & (0,0),  (1,0)$\ppp$, (0,1), (2,0),\\
	& 	&  (0,-2), (1,-2), (-2,0)$\pp$, (2,-2),  &  (1,-1), (2,-1), (-1,1)$\pp$, (3,-1)\\
	&	&  (-3,0), (-3,1), (3,-3) & (-2,1), (-2,2), (4,-2)\\
\hline
	& 5 & (0,0) & (1,1)\\
\hline
(0,0) & 2 & (-1,-1)$^{\oplus 3}$,  (0,-1)$\ppp$, (-1,0), (1,-1), & (0,0)$^{\oplus 3}$, (1,0)$\ppp$, (0,1), (2,0),  \\
	& 	&    (0,-2), (1,-2)$\pp$, (-2,0)$\ppp$,  (-2,-1),  & (1,-1), (2,-1)$\pp$,  (-1,1)$\ppp$,  (-1,0),   \\
	&	&   (2,-2), (-3,0)$\pp$, (-3,1), (3,-3), & (3,-1),  (-2,1)$\pp$, (-2,2), (4,-2), \\
	&     & (-4,1), (-5,1) & (-3,2), (-4,2)\\
\hline
	& 4  & (0,0), (0,-1), (-1,0), (1,-1), & (1,1), (1,0), (0,1), (2,0)\\
	&     &  (-2,0), (-2,1), (-3,1) & (-1,1), (-1,2), (-2,2)\\
\hline
\end{tabular}

\newpage
\begin{center} {\bf Alcove 4}
\end{center}

\medskip\noindent
\begin{tabular}{|c|c|c|c|}
\hline
$\s_0$ & $j$ & $\s_1$ for $\widehat{Z}_1'(1,2)$ & $\s_1$ for $\widehat{Z}_1'(8,9)$\\
\hline
(5,5) & 3 & (0,-1) & (1,0)\\
\hline
(3,5) & 2 & (0,-1), (-2,0) & (1,0), (-1,1)\\
\hline
(4,4) & 3 & (0,-1), (-2,0) & (1,0), (-1,1)\\
\hline
	& 5 & (-1,0) & (0,1)\\
\hline
(4,3) & 0 & (-1,-1) & (0,0)\\
\hline
	& 2 & (0,-1), (-2,0) & (1,0), (-1,1)\\
\hline
	& 4 & (-1,0) & (0,1)\\
\hline
(3,3)  & 1 & (-1,-1) & (0,0)\\
\hline
	& 3 & (0,-1), (-2,0) & (1,0), (-1,1)\\
\hline
	& 5 & (-1,0), (1,-1) & (0,1), (2,0)\\
\hline
(5,1) & 2 & (-1,-1) & (0,0)\\
\hline
	& 4 & (0,-1), (-1,0), (1,-1), (-2,0), (-3,1) & (1,0), (0,1), (2,0), (-1,1), (-2,2)\\
\hline
(0,4) & 2 & (0,-1), (-1,-1), (1,-2), (-2,0), & (1,0), (0,0), (2,-1), (-1,1),\\
\hline
	& 4 & (-1,0), (1,-1) & (0,1), (2,0)\\
\hline
(2,2) & 1 & (0,-1), (-1,-1)$^{\oplus 2}$, (1,-2), (-2,0), (-3,0) & (1,0), (0,0)$^{\oplus 2}$, (2,-1), (-1,1) (-2,1),\\
\hline
	& 3 & (0,-1), (-1,0)$^{\oplus 2}$, (1,-1), (-2,0), (-3,1) & (1,0), (0,1)$^{\oplus 2}$, (2,0), (-1,1) (-2,2),\\
\hline
(1,2) & 2  & (0,-1), (-1,-1)$^{\oplus 2}$, (1,-2), (-2,0), (-3,0) & (1,0), (0,0)$^{\oplus 2}$, (2,-1), (-1,1) (-2,1),\\
\hline
	& 4 &  (0,-1)$^{\oplus 2}$, (-1,0)$^{\oplus 3}$, (1,-1)$^{\oplus 2}$, & (1,0)$^{\oplus 2}$, (0,1)$^{\oplus 3}$, (2,0)$^{\oplus 2}$, \\
	&  &  (-2,0)$^{\oplus 2}$, (2,-2), (-3,1) & (-1,1)$^{\oplus 2}$, (3,-1), (-2,2),\\
\hline
	& 6 & (0,0) & (1,1)\\
\hline
(1,1) & 1  & (0,-1), (-1,-1)$^{\oplus 2}$, (0,-2),  & (1,0), (0,0)$^{\oplus 2}$, (1,-1), \\
	&  &  (1,-2), (-2,0), (-3,0) & (2,-1), (-1,1) (-2,1)\\
\hline
	& 3  &  (0,-1)$^{\oplus 2}$, (-1,0)$^{\oplus 2}$, (1,-1), & (1,0)$^{\oplus 2}$, (0,1)$^{\oplus 2}$, (2,0), \\
	&  &  (-2,0)$^{\oplus 2}$, (2,-2), (-3,1) & (-1,1)$^{\oplus 2}$, (3,-1), (-2,2),\\
\hline
	& 5  & (0,0) & (1,1)\\
\hline
(2,0) & 2 & (0,-1), (-1,-1), (0,-2),  & (1,0), (0,0), (1,-1), \\
	&  &  (1,-2), (-2,0), (-3,0) & (2,-1), (-1,1) (-2,1)\\
\hline
	& 4 &  (0,0), (0,-1)$^{\oplus 2}$, (-1,0)$^{\oplus 2}$, (1,-1), & (1,1), (1,0)$^{\oplus 2}$, (0,1)$^{\oplus 2}$, (2,0), \\
	&  &  (-2,0), (2,-2), (-3,1), (3,-2) & (-1,1), (3,-1), (-2,2), (4,-1)\\
\hline
(0,0) & 1 & (0,-1), (-1,-1), (0,-2),  (-2,-1) & (1,0), (0,0), (1,-1), (-1,0)\\
	&  &  (1,-2), (-2,0), (-3,0) & (2,-1), (-1,1) (-2,1)\\
\hline
	& 3 & (0,0), (-1,-1), (0,-1)$^{\oplus 3}$, (-1,0)$^{\oplus 3}$,    & (1,1), (0,0), (1,0)$^{\oplus 3}$, (0,1)$^{\oplus 3}$,  \\
	& 	&  (1,-1)$^{\oplus 2}$, (1,-2), (-2,0)$^{\oplus 3}$,   &  (2,0)$^{\oplus 2}$,   (2,-1), (-1,1)$^{\oplus 3}$, \\
	&      & (2,-2),  (-2,1), (3,-2),  (-3,0), & (3,-1), (-1,2), (4,-1),  (-2,1),\\
	&      & (-3,1)$^{\oplus 2}$, (-4,1),  (-5,2) &    (-2,2)$^{\oplus 2}$, (-3,2), (-4,3)\\
\hline
\end{tabular}
\end{center}

\newpage
\begin{center} {\bf Alcove 5}

\medskip\noindent
\begin{tabular}{|c|c|c|c|}
\hline
$\s_0$ & $j$ & $\s_1$ for $\widehat{Z}_1'(2,2)$ & $\s_1$ for $\widehat{Z}_1'(9,9)$\\
\hline
(5,5) & 2 & (0,-1) & (1,0)\\
\hline
(3,5) & 3 & (0,-1) & (1,0)\\
\hline
	& 5 & (-1,0) & (0,1)\\
\hline
(4,4) & 2 & (0,-1), (-2,0) & (1,0), (-1,1)\\
\hline
	& 4 & (-1,0)  & (0,1)\\
\hline
(4,3) & 3 & (0,-1), (-2,0) & (1,0), (-1,1)\\
\hline
	& 5 & (-1,0), (1,-1) & (0,1), (2,0)\\
\hline
(3,3)  & 0 & (-1,-1) & (0,0)\\
\hline
	& 2 & (0,-1), (-2,0) & (1,0), (-1,1)\\
\hline
	& 4 & (-1,0), (1,-1) & (0,1), (2,0)\\
\hline
(5,1) & 1 & (-1,-1) & (0,0)\\
\hline
	& 3 & (0,-1), (-1,0), (1,-1), (-2,0), (-3,1) & (1,0), (0,1), (2,0), (-1,1), (-2,2)\\
\hline
(0,4) & 1& (-1,-1) & (0,0)\\
\hline
	& 3 & (0,-1), (-1,0), (1,-1), (-2,0), (2,-2) & (1,0), (0,1), (2,0), (-1,1), (3,-1)\\
\hline
(2,2) & 2 & (-1,-1) & (0,0)\\
\hline
	& 4 &  (0,-1)$^{\oplus 2}$, (-1,0)$\pp$, (1,-1)$\pp$,&  (1,0)$^{\oplus 2}$, (0,1)$\pp$, (2,0)$\pp$,\\
	&	& (-2,0)$\pp$, (2,-2), (-3,1) &  (-1,1)$\pp$, (3,-1), (-2,2)\\
\hline
	& 6 & (0,0) & (1,1)\\
\hline
(1,2) & 1  & (0,-1), (-1,-1)$^{\oplus 2}$, (1,-2), (-2,0), (-3,0) & (1,0), (0,0)$^{\oplus 2}$, (2,-1), (-1,1) (-2,1),\\
\hline
	& 3 &  (0,-1)$^{\oplus 2}$, (-1,0)$^{\oplus 3}$, (1,-1)$^{\oplus 2}$, & (1,0)$^{\oplus 2}$, (0,1)$^{\oplus 3}$, (2,0)$^{\oplus 2}$, \\
	&  &  (-2,0)$^{\oplus 2}$, (2,-2), (-3,1) & (-1,1)$^{\oplus 2}$, (3,-1), (-2,2),\\
\hline
	& 5 & (0,0) & (1,1)\\
\hline
(1,1) & 2  & (0,-1), (-1,-1), (1,-2),  (-2,0), (-3,0) & (1,0), (0,0), (2,-1), (-1,1), (-2,1)\\
\hline
	& 4  & (0,0),  (0,-1)$^{\oplus 2}$, (-1,0)$^{\oplus 3}$, (1,-1)$\pp$, & (1,1), (1,0)$^{\oplus 2}$, (0,1)$^{\oplus 3}$, (2,0)$\pp$, \\
	&  &  (-2,0), (2,-2), (3,-2), (-3,1)  & (-1,1), (3,-1), (4,-1), (-2,2),\\
\hline
(2,0) & 1 & (0,-1), (-1,-1), (0,-2),  & (1,0), (0,0), (1,-1), \\
	&  &  (1,-2), (-2,0), (-3,0) & (2,-1), (-1,1) (-2,1)\\
\hline
	& 3 &  (0,0), (0,-1)$^{\oplus 2}$, (-1,0)$^{\oplus 2}$, (1,-1), & (1,1), (1,0)$^{\oplus 2}$, (0,1)$^{\oplus 2}$, (2,0), \\
	&  &  (-2,0), (2,-2), (3,-2),  (-3,1)  & (-1,1), (3,-1), (4,-1), (-2,2)\\
\hline
(0,0) & 2 & (0,-1), (-1,-1), (0,-2),   & (1,0), (0,0), (1,-1),\\
	&  &  (1,-2), (-2,0), (-3,0) & (2,-1), (-1,1) (-2,1)\\
\hline
	& 4 & (0,0)$\pp$, (-1,-1), (0,-1)$^{\oplus 3}$, (-1,0)$^{\oplus 4}$,    & (1,1)$\pp$, (0,0), (1,0)$^{\oplus 3}$, (0,1)$^{\oplus 4}$,  \\
	& 	&  (1,-1)$^{\oplus 2}$, (2,-1), (-2,0)$^{\oplus 2}$,   &  (2,0)$^{\oplus 2}$,   (3,0), (-1,1)$^{\oplus 2}$, \\
	&      & (2,-2)$\pp$,  (-2,1), (3,-2) & (3,-1)$\pp$, (-1,2), (4,-1)\\
	&      & (-3,1)$^{\oplus 2}$, (-4,1), (-4,2) &    (-2,2)$^{\oplus 2}$, (-3,2), (-3,3)\\
\hline
\end{tabular}
\end{center}

\newpage
\begin{center} {\bf Alcove 6}

\medskip\noindent
\begin{tabular}{|c|c|c|c|}
\hline
$\s_0$ & $j$ & $\s_1$ for $\widehat{Z}_1'(0,4)$ & $\s_1$ for $\widehat{Z}_1'(7,11)$\\
\hline
(5,5) & 5 & (-1,0) & (0,1)\\
\hline
(3,5) & 2 & (0,-1) & (1,0)\\
\hline
	& 4 & (-1,0) & (0,1)\\
\hline
(4,4) & 1 & (0,-1), (-2,0) & (1,0), (-1,1)\\
\hline
	& 3 & (-1,0)  & (0,1)\\
\hline
(4,3) & 2 & (0,-1), (-2,0) & (1,0), (-1,1)\\
\hline
	& 4 & (-1,0), (1,-1) & (0,1), (2,0)\\
\hline
(3,3)  & 3 & (0,-1), (-1,0), (1,-1), (-2,0), (-3,1) & (1,0), (0,1), (2,0), (-1,1), (-2,2)\\
\hline
(5,1) & 0 & (-1,-1) & (0,0)\\
\hline
	& 2 & (0,-1), (-1,0), (1,-1), (-2,0), (-3,1) & (1,0), (0,1), (2,0), (-1,1), (-2,2)\\
\hline
(0,4) & 4 & (0,-1), (-1,0), (1,-1), (-2,0), (-3,1) & (1,0), (0,1), (2,0), (-1,1), (-2,2)\\
\hline
	& 6 & (0,0) & (1,1)\\
\hline
(2,2) & 1 & (-1,-1) & (0,0)\\
\hline
	& 3 &  (0,-1)$^{\oplus 2}$, (-1,0)$\pp$, (1,-1)$\pp$,&  (1,0)$^{\oplus 2}$, (0,1)$\pp$, (2,0)$\pp$,\\
	&	& (-2,0)$\pp$, (2,-2), (-3,1) &  (-1,1)$\pp$, (3,-1), (-2,2)\\
\hline
	& 5 & (0,0) & (1,1)\\
\hline
(1,2) & 2 &  (0,-1)$^{\oplus 2}$, (-1,-1), (-1,0)$^{\oplus 2}$, (1,-1), & (1,0)$^{\oplus 2}$, (0,0), (0,1)$^{\oplus 2}$, (2,0), \\
	&  &  (-2,0)$^{\oplus 2}$, (2,-2), (-3,1), (-4,1) & (-1,1)$^{\oplus 2}$, (3,-1), (-2,2), (-3,2) \\
\hline
	& 4  & (0,0), (0,-1), (-1,0), (1,-1), & (1,1),  (1,0), (0,1), (2,0)\\
	&	& (-2,0), (-2,1), (-3,1) &  (-1,1), (-1,2), (-2,2)\\
\hline
(1,1) & 3  &  (0,-1)$\ppp$, (-1,-1), (-1,0)$\pp$, (1,-1) & (1,0)$\ppp$, (0,0), (0,1)$\pp$, (2,0)\\
	& 	& (-2,0)$\pp$, (2,-2), (-3,1), (-4,1) & (-1,1)$\pp$, (3,-1), (-2,2), (-3,2)\\
\hline
	&  5  & (0,0), (-1,0), (1,-1), (2,-1), (-2,1) & (1,1), (0,1), (2,0), (3,0), (-1,2)\\
\hline
(2,0) & 2 & (0,-1), (-2,0) & (1,0), (-1,1)\\
\hline
	& 4 & (0,0), (0,-1), (-1,-1), (-1,0)$\ppp,$ & (1,1), (1,0), (0,0), (0,1)$\ppp$\\
	&	&  (1,-1), (2,-1), (-2,0), (2,-2), & (2,0), (3,0), (-1,1), (3,-1)\\
	&	&  (-2,1), (-3,1), (-4,1), (-4,2) & (-1,2), (-2,2), (-3,2), (-3,3)\\
\hline
(0,0) & 1 & (0,-1), (-1,-1), (0,-2),   & (1,0), (0,0), (1,-1),\\
	&  &  (1,-2), (-2,0), (-3,0) & (2,-1), (-1,1) (-2,1)\\
\hline
	& 3 & (0,0)$\pp$, (-1,-1), (0,-1)$^{\oplus 3}$, (-1,0)$^{\oplus 4}$,    & (1,1)$\pp$, (0,0), (1,0)$^{\oplus 3}$, (0,1)$^{\oplus 4}$,  \\
	& 	&  (1,-1)$^{\oplus 2}$, (2,-1), (-2,0)$^{\oplus 2}$,   &  (2,0)$^{\oplus 2}$,   (3,0), (-1,1)$^{\oplus 2}$, \\
	&      & (2,-2)$\pp$,  (-2,1), (3,-2) & (3,-1)$\pp$, (-1,2), (4,-1)\\
	&      & (-3,1)$^{\oplus 2}$, (-4,1), (-4,2) &    (-2,2)$^{\oplus 2}$, (-3,2), (-3,3)\\
\hline
\end{tabular}
\end{center}

\newpage
\begin{center} 

{\bf Alcove 7}

\medskip\noindent
\begin{tabular}{|c|c|c|c|}
\hline
$\s_0$ & $j$ & $\s_1$ for $\widehat{Z}_1'(5,1)$ & $\s_1$ for $\widehat{Z}_1'(12,8)$\\
\hline
(5,5) & 1 & (0,-1) & (1,0)\\
\hline
(3,5) & 2 & (0,-1) & (1,0)\\
\hline
	& 4 & (-1,0) & (0,1)\\
\hline
(4,4) & 3 & (0,-1)  & (1,0)\\
\hline
	& 5 & (-1,0), (1,-1) & (0,1), (2,0)\\
\hline
(4,3) & 2 & (0,-1), (-2,0) & (1,0), (-1,1)\\
\hline
	& 4 & (-1,0), (1,-1) & (0,1), (2,0)\\
\hline
(3,3)  & 3 & (0,-1), (-1,0), (1,-1), (-2,0), (2,-2) & (1,0), (0,1), (2,0), (-1,1), (3,-1)\\
\hline
(5,1) & 4 & (0,-1), (-1,0), (1,-1), (-2,0), (2,-2) & (1,0), (0,1), (2,0), (-1,1), (3,-1)\\
\hline
	& 6 & (0,0) & (1,1)\\
\hline
(0,4) & 0 & (-1,-1) & (0,0)\\
\hline
	& 2  & (0,-1), (-1,0), (1,-1), (-2,0), (2,-2) & (1,0), (0,1), (2,0), (-1,1), (3,-1)\\
\hline
(2,2) & 1 & (-1,-1) & (0,0)\\
\hline
	& 3 &  (0,-1)$^{\oplus 2}$, (-1,0)$\pp$, (1,-1)$\pp$,&  (1,0)$^{\oplus 2}$, (0,1)$\pp$, (2,0)$\pp$,\\
	&	& (-2,0)$\pp$, (2,-2), (-3,1) &  (-1,1)$\pp$, (3,-1), (-2,2)\\
\hline
	& 5 & (0,0) & (1,1)\\
\hline
(1,2) & 2  & (0,-1), (-1,0), (-1,-1), (1,-1), & (1,0), (0,1), (0,0), (2,0)\\
	&	& (1,-2), (-2,0), (2,-2) & (2,-1), (-1,1), (3,-1)\\
\hline
	& 4  & (0,0),  (0,-1)$^{\oplus 2}$, (-1,0)$^{\oplus 2}$, (1,-1)$\pp$, & (1,1), (1,0)$^{\oplus 2}$,  (0,1)$^{\oplus 2}$, (2,0)$\pp$, \\
	&  &  (-2,0), (2,-2), (3,-2),  (-3,1)  & (-1,1), (3,-1), (4,-1),  (-2,2) \\	
\hline
(1,1) & 1 & (0,-1), (-1,-1), (1,-2), (-2,0), (-3,0) & (1,0), (0,0), (2,-1), (-1,1), (-2,1)\\
\hline
	& 3  & (0,0),  (0,-1)$\pp$, (-1,0)$\ppp$, (1,-1)$\pp$ & (1,1), (1,0)$\pp$,  (0,1)$\ppp$, (2,0)$\pp$\\
	& 	& (-2,0), (2,-2), (3,-2), (-3,1) & (-1,1), (3,-1), (4,-1), (-2,2)\\
\hline
(2,0) & 2 & (0,0), (0,-1)$\ppp$, (-1,-1), (-1,0) & (1,1), (1,0)$\ppp$, (0,0), (0,1)\\
	&	&  (1,-1), (-2,0), (1,-2), (2,-2), & (2,0),  (-1,1), (2,-1), (3,-1)\\
	&	&  (3,-2), (3,-3), (-3,0), (-3,1) & (4,-1), (4,-2), (-2,1), (-2,2)\\	
\hline
	 &  4  & (-1,0), (1,-1) & (0,1), (2,0)\\
\hline
(0,0) & 3 & (0,0), (-1,-1)$\pp$, (0,-1)$^{\oplus 4}$, (-1,0)$^{\oplus 3}$,    & (1,1), (0,0)$\pp$, (1,0)$^{\oplus 4}$, (0,1)$^{\oplus 3}$,  \\
	& 	&  (1,-1)$^{\oplus 2}$, (1,-2), (-2,0)$^{\oplus 2}$,   &  (2,0)$^{\oplus 2}$,   (2,-1), (-1,1)$^{\oplus 2}$, \\
	&      & (2,-2)$\pp$, (3,-2), (3,-3), & (3,-1)$\pp$, (4,-1), (4,-2),\\
	&      & (-3,0), (-3,1)$^{\oplus 2}$, (-4,1) &   (-2,1),  (-2,2)$^{\oplus 2}$, (-3,2)\\
\hline
	& 5  & (0,0), (-1,0), (1,-1) & (1,1), (0,1), (2,0)\\
	&	& (-1,1), (2,-1), (-2,1) & (0,2), (3,0), (-1,2)\\
\hline
\end{tabular}
\end{center}

\newpage
\begin{center} {\bf Alcove 8}

\medskip\noindent
\begin{tabular}{|c|c|c|c|}
\hline
$\s_0$ & $j$ & $\s_1$ for $\widehat{Z}_1'(3,3)$ & $\s_1$ for $\widehat{Z}_1'(10,10)$\\
\hline
(5,5) & 4 & (-1,0) & (0,1)\\
\hline
(3,5) & 1 & (0,-1) & (1,0)\\
\hline
	& 3 & (-1,0) & (0,1)\\
\hline
(4,4) & 2 & (0,-1)  & (1,0)\\
\hline
	& 4 & (-1,0), (1,-1) & (0,1), (2,0)\\
\hline
(4,3) & 1 & (0,-1), (-2,0) & (1,0), (-1,1)\\
\hline
	& 3 & (-1,0), (1,-1) & (0,1), (2,0)\\
\hline
(3,3)  & 2 & (0,-1), (-2,0) & (1,0), (-1,1)\\
\hline
	& 4 & (-1,0), (1,-1) & (0,1), (2,0)\\
\hline
	& 6 & (0,0) & (1,1)\\
\hline
(5,1) & 3 & (0,-1), (-1,0), (1,-1), (-2,0), (2,-2) & (1,0), (0,1), (2,0), (-1,1), (3,-1)\\
\hline
	& 5 & (0,0) & (1,1)\\
\hline
(0,4) & 3 & (0,-1), (-1,0), (1,-1), (-2,0), (-3,1) & (1,0), (0,1), (2,0), (-1,1), (-2,2)\\
\hline
	& 5 & (0,0) & (1,1)\\
\hline
(2,2) & 0 & (-1,-1) & (0,0)\\
\hline
	& 2 &  (0,-1)$^{\oplus 2}$, (-1,0)$\pp$, (1,-1)$\pp$,&  (1,0)$^{\oplus 2}$, (0,1)$\pp$, (2,0)$\pp$,\\
	&	& (-2,0)$\pp$, (2,-2), (-3,1) &  (-1,1)$\pp$, (3,-1), (-2,2)\\
\hline
	& 4 & (0,0) & (1,1)\\
\hline
(1,2) & 1 & (-1,-1) & (0,0)\\
\hline
	& 3  & (0,-1)$^{\oplus 3}$, (-1,0)$^{\oplus 2}$, (1,-1)$\pp$, & (1,0)$^{\oplus 3}$,  (0,1)$^{\oplus 2}$, (2,0)$\pp$, \\
	&  &  (-2,0)$\pp$, (2,-2),  (-3,1)  & (-1,1)$\pp$, (3,-1),  (-2,2) \\	
\hline
	& 5 & (0,0)$\pp$, (-1,0), (1,-1), (2,-1), (-2,1) & (1,1)$\pp$, (0,1), (2,0), (3,0), (-1,2)\\
\hline
(1,1) & 2  & (0,-1)$\ppp$, (-1,-1), (-1,0)$\pp$, (1,-1)  & (1,0)$\ppp$, (0,0), (0,1)$\pp$, (2,0)\\
	& 	& (-2,0)$\pp$, (2,-2),  (-3,1), (-4,1) & (-1,1)$\pp$, (3,-1), (-2,2), (-3,2)\\
\hline
	& 4  & (0,0), (-1,0), (1,-1), (2,-1), (-2,1) & (1,1), (0,1), (2,0), (3,0), (-1,2)\\
\hline
(2,0) & 3 & (0,-1)$\pp$, (-1,-1), (-1,0)$\pp$, (1,-1) & (1,0)$\pp$, (0,0), (0,1)$\pp$, (2,0)\\
	&	&  (-2,0), (2,-2), (-3,1), (-4,1) & (-1,1), (3,-1), (-2,2), (-3,2)\\
\hline
	 &  5  & (0,0), (-1,0), (1,-1), (-1,1) & (1,1),  (0,1), (2,0), (0,2)\\
	 &	& (2,-1), (-2,1) & (3,0), (-1,2)\\
\hline
(0,0) & 2 & (0,0), (-1,-1)$\pp$, (0,-1)$^{\oplus 4}$, (-1,0)$^{\oplus 3}$,    & (1,1), (0,0)$\pp$, (1,0)$^{\oplus 4}$, (0,1)$^{\oplus 3}$,  \\
	& 	&  (1,-1)$^{\oplus 2}$, (1,-2), (-2,0)$^{\oplus 2}$,   &  (2,0)$^{\oplus 2}$,   (2,-1), (-1,1)$^{\oplus 2}$, \\
	&      & (2,-2)$\pp$, (3,-2), (3,-3), & (3,-1)$\pp$, (4,-1), (4,-2),\\
	&      & (-3,0), (-3,1)$^{\oplus 2}$, (-4,1) &   (-2,1),  (-2,2)$^{\oplus 2}$, (-3,2)\\
\hline
	& 4  & (0,0), (-1,0), (1,-1), & (1,1), (0,1), (2,0)\\
	&	& (-1,1), (2,-1), (-2,1) & (0,2), (3,0), (-1,2)\\
\hline
\end{tabular}
\end{center}

\newpage
\begin{center} {\bf Alcove 11}

\medskip\noindent
\begin{tabular}{|c|c|c|c|}
\hline
$\s_0$ & $j$ & $\s_1$ for $\widehat{Z}_1'(4,3)$ & $\s_1$ for $\widehat{Z}_1'(11,10)$\\
\hline
(5,5) & 3 & (-1,0) & (0,1)\\
\hline
(3,5) & 4 &  (-1,0), (1,-1)  & (0,1), (2,0)\\
\hline
(4,4) & 1 & (0,-1)  & (1,0)\\
\hline
	& 3 & (-1,0), (1,-1) & (0,1), (2,0)\\
\hline
(4,3) & 2 & (0,-1)  & (1,0)\\
\hline
	& 4  &  (-1,0), (1,-1) & (0,1), (2,0)\\
\hline
	& 6  & (0,0)  & (1,1)\\
\hline
(3,3)  & 1 & (0,-1), (-2,0) & (1,0), (-1,1)\\
\hline
	& 3 & (-1,0), (1,-1) & (0,1), (2,0)\\
\hline
	& 5 & (0,0) & (1,1)\\
\hline
(5,1) & 2 & (0,-1), (-1,0), (1,-1),  (-2,0), (2,-2) & (1,0), (0,1), (2,0), (-1,1), (3,-1)\\
\hline
	& 4 & (0,0)  & (1,1)\\
\hline
(0,4) & 2 & (0,-1), (-2,0) & (1,0), (-1,1)\\
\hline
	& 4 & (0,0), (-1,0), (1,-1), (-2,1)  & (1,1), (0,1), (2,0), (-1,2)\\
\hline
(2,2) &  3 &  (0,-1)$\pp$, (-1,0), (1,-1), (-2,0), (2,-2)  & (1,0)$\pp$, (0,1), (2,0), (-1,1), (3,-1)\\
\hline
	& 5 & (0,0)$\pp$, (-1,0), (1,-1), (2,-1), (-2,1)  & (1,1)$\pp$, (0,1), (2,0), (3,0), (-1,2)\\
\hline
(1,2) & 0 & (-1,-1) & (0,0)\\
\hline
	& 2  & (0,-1)$^{\oplus 3}$, (-1,0)$^{\oplus 2}$, (1,-1)$\pp$, & (1,0)$^{\oplus 3}$,  (0,1)$^{\oplus 2}$, (2,0)$\pp$, \\
	&  &  (-2,0)$\pp$, (2,-2),  (-3,1)  & (-1,1)$\pp$, (3,-1),  (-2,2) \\	
\hline
	& 4 & (0,0)$\pp$, (-1,0), (1,-1), (2,-1), (-2,1) & (1,1)$\pp$, (0,1), (2,0), (3,0), (-1,2)\\
\hline
(1,1) &  1  & (-1,-1)  &  (0,0)\\
\hline
	& 3  & (0,-1)$\pp$,  (-1,0)$\pp$, (1,-1)$\pp$  & (1,0)$\pp$,  (0,1)$\pp$, (2,0)$\pp$\\
	& 	& (-2,0), (2,-2),  (-3,1)  & (-1,1), (3,-1), (-2,2)\\
\hline
	& 5  & (0,0)$\pp$, (-1,0), (1,-1), (-1,1),  & (1,1)$\pp$, (0,1), (2,0), (0,2)\\
	&    &  (2,-1), (-2,1) &  (3,0), (-1,2)\\
\hline
(2,0) & 2 & (0,-1)$\pp$, (-1,-1), (-1,0)$\pp$, (1,-1) & (1,0)$\pp$, (0,0), (0,1)$\pp$, (2,0)\\
	&	&  (-2,0), (2,-2), (-3,1), (-4,1) & (-1,1), (3,-1), (-2,2), (-3,2)\\
\hline
	 &  4  & (0,0), (-1,0), (1,-1), (-1,1) & (1,1),  (0,1), (2,0), (0,2)\\
	 &	& (2,-1), (-2,1) & (3,0), (-1,2)\\
\hline
(0,0) & 3 & (0,0), (-1,-1), (0,-1)$^{\oplus 3}$, (-1,0)$^{\oplus 3}$,    & (1,1), (0,0), (1,0)$^{\oplus 3}$, (0,1)$^{\oplus 3}$,  \\
	& 	&  (1,-1)$^{\oplus 3}$, (2,-1),  (1,-2), (-2,0)$^{\oplus 2}$,   &  (2,0)$^{\oplus 3}$,  (3,0),  (2,-1), (-1,1)$^{\oplus 2}$, \\
	&      & (2,-2)$\pp$, (-2,1),  (3,-2), & (3,-1)$\pp$, (-1,2), (4,-1), \\
	&      & (4,-3), (-3,1), (-4,1) &   (5,-2),  (-2,2), (-3,2)\\
\hline
	& 5  & (1,0),  (0,0), (-1,0), (1,-1), & (2,1),  (1,1), (0,1), (2,0)\\
	&	& (-1,1), (2,-1), (-2,1) & (0,2), (3,0), (-1,2)\\
\hline
\end{tabular}
\end{center}

\newpage
\begin{center} {\bf Alcove 13}

\medskip\noindent
\begin{tabular}{|c|c|c|c|}
\hline
$\s_0$ & $j$ & $\s_1$ for $\widehat{Z}_1'(4,4)$ & $\s_1$ for $\widehat{Z}_1'(11,11)$\\
\hline
(5,5) & 2 & (-1,0) & (0,1)\\
\hline
(3,5) & 3 &  (-1,0), (1,-1)  & (0,1), (2,0)\\
\hline
(4,4) & 4 &  (-1,0), (1,-1) & (0,1), (2,0)\\
\hline
	& 6 & (0,0) & (1,1)\\
\hline
(4,3) & 1 & (0,-1)  & (1,0)\\
\hline
	& 3  &  (-1,0), (1,-1) & (0,1), (2,0)\\
\hline
	& 5  & (0,0)  & (1,1)\\
\hline
(3,3)  & 2 & (0,-1) & (1,0)\\
\hline
	& 4 & (0,0),  (-1,0), (1,-1), (-2,1)  &  (1,1), (0,1), (2,0), (-1,2)\\
\hline
(5,1) & 3 & (0,-1)  & (1,0)\\
\hline
	& 5 &  (0,0), (-1,0), (1,-1), (2,-1), (-2,1)  & (1,1),  (0,1), (2,0), (3,0), (-1,2)\\
\hline
(0,4) & 1 & (0,-1), (-2,0) & (1,0), (-1,1)\\
\hline
	& 3 & (0,0), (-1,0), (1,-1), (-2,1)  & (1,1), (0,1), (2,0), (-1,2)\\
\hline
(2,2) &  2 &  (0,-1)$\pp$, (-1,0), (1,-1), (-2,0), (2,-2)  & (1,0)$\pp$, (0,1), (2,0), (-1,1), (3,-1)\\
\hline
	& 4 & (0,0)$\pp$, (-1,0), (1,-1), (2,-1), (-2,1)  & (1,1)$\pp$, (0,1), (2,0), (3,0), (-1,2)\\
\hline
(1,2) & 3  & (0,0), (0,-1)$^{\oplus 2}$, (-1,0)$^{\oplus 2}$, (1,-1)$\pp$, & (1,1), (1,0)$^{\oplus 2}$,  (0,1)$^{\oplus 2}$, (2,0)$\pp$, \\
	&  &  (-2,0), (-2,1),  (2,-2),  (-3,1)  & (-1,1), (-1,2),  (3,-1),  (-2,2) \\	
\hline
	& 5 & (0,0)$\pp$, (-1,0), (1,-1), (-1,1),  &  (1,1)$\pp$, (0,1),  (2,0), (0,2)\\
	&    & (2,-1), (-2,1) &  (3,0), (-1,2)\\
\hline
(1,1) &  0  & (-1,-1)  &  (0,0)\\
\hline
	& 2  & (0,-1)$\pp$,  (-1,0)$\pp$, (1,-1)$\pp$  & (1,0)$\pp$,  (0,1)$\pp$, (2,0)$\pp$\\
	& 	& (-2,0), (2,-2),  (-3,1)  & (-1,1), (3,-1), (-2,2)\\
\hline
	& 4  & (0,0)$\pp$, (-1,0), (1,-1), (-1,1),  & (1,1)$\pp$, (0,1), (2,0), (0,2)\\
	&    &  (2,-1), (-2,1) &  (3,0), (-1,2)\\
\hline
(2,0) & 1 & (-1,-1) & (0,0)\\
\hline
	& 3 & (0,0), (0,-1), (-1,0)$\ppp$, (1,-1)$\pp$,   & (1,1), (1,0), (0,1)$\ppp$, (2,0)$\pp$, \\
	&   & (-1,1), (2,-1), (-2,0), (2,-2),	&  (0,2), (3,0),  (-1,1), (3,-1),\\
	&	&  (-2,1),  (-3,1), (-4,2) & (-1,2), (-2,2), (-3,3)\\
\hline
	 &  5  & (0,0) & (1,1)\\
\hline
(0,0) &  2 & (-1,-1), (0,-1), (-1,0), (1,-1), & (0,0), (1,0), (0,1), (2,0)\\
	&	& (1,-2), (2,-2), (-2,0) & (2,-1), (3,-1), (-1,1)\\
\hline
	&  4 & (1,0), (0,0)$\ppp$, (-1,1),  (0,-1),     & (2,1), (1,1)$\ppp$, (0,2),  (1,0),   \\
	& 	&  (-1,0)$^{\oplus 3}$, (1,-1)$^{\oplus 3}$, (2,-1)$\pp$,  (-2,0),   &  (0,1)$^{\oplus 3}$, (2,0)$^{\oplus 3}$,  (3,0)$\pp$,   (-1,1), \\
	&      & (2,-2), (-2,1)$\pp$, (3,-2), & (3,-1), (-1,2)$\pp$, (4,-1), \\
	&      & (4,-2), (-3,1), (-4,2) &   (5,-1),  (-2,2), (-3,3)\\
\hline
\end{tabular}
\end{center}

\newpage
\begin{center} {\bf Alcove 15}

\medskip\noindent
\begin{tabular}{|c|c|c|c|}
\hline
$\s_0$ & $j$ & $\s_1$ for $\widehat{Z}_1'(3,5)$ & $\s_1$ for $\widehat{Z}_1'(10,12)$\\
\hline
(5,5) & 1 & (-1,0) & (0,1)\\
\hline
(3,5) & 2 &  (-1,0)  & (0,1) \\
\hline
	& 6 & (0,0)  & (1,1)\\
\hline
(4,4) & 3 &  (-1,0), (1,-1) & (0,1), (2,0)\\
\hline
	& 5 & (0,0) & (1,1)\\
\hline
(4,3) & 4  & (0,0),  (-1,0), (1,-1), (-2,1) & (1,1), (0,1), (2,0), (-1,2)\\
\hline
(3,3)  & 1 & (0,-1) & (1,0)\\
\hline
	& 3 & (0,0),  (-1,0), (1,-1), (-2,1)  &  (1,1), (0,1), (2,0), (-1,2)\\
\hline
(5,1) & 2 & (0,-1)  & (1,0)\\
\hline
	& 4 &  (0,0), (-1,0), (1,-1), (2,-1), (-2,1)  & (1,1),  (0,1), (2,0), (3,0), (-1,2)\\
\hline
(0,4) &  2 & (0,0), (0,-1), (-1,0), (1,-1), & (1,1), (1,0), (0,1), (2,0),\\
	& 	& (-2,1), (-3,1)  &  (-1,2), (-2,2)\\
\hline
(2,2) &  3 &  (0,0), (0,-1), (-1,0), (1,-1),   & (1,1), (1,0), (0,1), (2,0), \\
	&	& (-2,1), (-3,1) & (-1,2), (-2,2)\\
\hline
	& 5 & (0,0), (-1,0), (1,-1), (-1,1),  & (1,1), (0,1), (2,0), (0,2)\\
	&	& (2,-1), (-2,1)  &  (3,0), (-1,2)\\
\hline
(1,2) & 2  & (0,0), (0,-1)$^{\oplus 2}$, (-1,0)$^{\oplus 2}$, (1,-1)$\pp$, & (1,1), (1,0)$^{\oplus 2}$,  (0,1)$^{\oplus 2}$, (2,0)$\pp$, \\
	&  &  (-2,0), (-2,1),  (2,-2),  (-3,1)  & (-1,1), (-1,2),  (3,-1),  (-2,2) \\	
\hline
	& 4 & (0,0)$\pp$, (-1,0), (1,-1), (-1,1),  &  (1,1)$\pp$, (0,1),  (2,0), (0,2)\\
	&    & (2,-1), (-2,1) &  (3,0), (-1,2)\\
\hline
(1,1) & 3 &  (0,0)$\pp$, (0,-1), (-1,0)$\ppp$, (1,-1)$\pp$, & 	(1,1)$\pp$, (1,0), (0,1)$\ppp$, (2,0)$\pp$	\\
	&	& (-1,1), (2,-1), (-2,0), (2,-2),  &	(0,2), (3,0), (-1,1), (3,-1)	\\
	&	& (-2,1)$\pp$, (-3,1), (-4,2) & 	(-1,2)$\pp$, (-2,2), (-3,3)	\\
\hline
	& 5 & (0,0) & (1,1)\\
\hline
(2,0) & 0 & (-1,-1) & (0,0)\\
\hline
	& 2 & (0,0), (0,-1), (-1,0)$\ppp$, (1,-1)$\pp$,   & (1,1), (1,0), (0,1)$\ppp$, (2,0)$\pp$, \\
	&   & (-1,1), (2,-1), (-2,0), (2,-2),	&  (0,2), (3,0),  (-1,1), (3,-1),\\
	&	&  (-2,1),  (-3,1), (-4,2) & (-1,2), (-2,2), (-3,3)\\
\hline
	 &  4  & (0,0) & (1,1)\\
\hline
(0,0) & 1 & (-1,-1)  & (0,0)\\
\hline
	&  3 & (0,0), (0,-1), (-1,0)$\ppp$, (1,-1)$\pp$,  &  (1,1), (1,0), (0,1)$\ppp$, (2,0)$\pp$, \\
	&	& (-1,1), (2,-1), (2,-2), (-2,0) & (0,2), (3,0), (3,-1), (-1,1)\\
	&	& (-2,1), (-3,1), (-4,2) & (-1,2), (-2,2), (-3,3)\\
\hline
	&  5 & (1,0), (0,0)$\ppp$, (-1,1),  (0,-1),     & (2,1), (1,1)$\ppp$, (0,2),  (1,0),   \\
	& 	&  (-1,0), (1,-1), (2,-1), (3,-1)  &  (0,1), (2,0),  (3,0),  (4,0)  \\
	&      & (-2,1), (3,-2),  (-3,1), (-3,2) & (-1,2), (4,-1), (-2,2), (-2,3)\\
\hline
\end{tabular}
\end{center}

\newpage
\begin{center} {\bf Alcove 16}

\medskip\noindent
\begin{tabular}{|c|c|c|c|}
\hline
$\s_0$ & $j$ & $\s_1$ for $\widehat{Z}_1'(5,5)$ & $\s_1$ for $\widehat{Z}_1'(12,12)$\\
\hline
(5,5) & 6 & (0,0) & (1,1)\\
\hline
(3,5) & 1 &  (-1,0)  & (0,1) \\
\hline
	& 5 & (0,0)  & (1,1)\\
\hline
(4,4) & 2 &  (-1,0), (1,-1) & (0,1), (2,0)\\
\hline
	& 4 & (0,0) & (1,1)\\
\hline
(4,3) & 3  & (0,0),  (-1,0), (1,-1), (-2,1) & (1,1), (0,1), (2,0), (-1,2)\\
\hline
(3,3)  & 4 & (0,0),  (-1,0), (1,-1), (2,-1), (-2,1) & (1,1), (0,1), (2,0), (3,0), (-1,2)\\
\hline
(5,1) & 1 & (0,-1)  & (1,0)\\
\hline
	& 3 &  (0,0), (-1,0), (1,-1), (2,-1), (-2,1)  & (1,1),  (0,1), (2,0), (3,0), (-1,2)\\
\hline
(0,4) &  5 & (0,0), (-1,0), (1,-1), (-1,1) & (1,1), (0,1), (2,0), (0,2)\\
	&	& (2,-1), (-2,1) & (3,0), (-1,2) \\
\hline
(2,2) &  2 &  (0,0), (0,-1), (-1,0), (1,-1),   & (1,1), (1,0), (0,1), (2,0), \\
	&	& (-2,1), (-3,1) & (-1,2), (-2,2)\\
\hline
	& 4 & (0,0), (-1,0), (1,-1), (-1,1),  & (1,1), (0,1), (2,0), (0,2)\\
	&	& (2,-1), (-2,1)  &  (3,0), (-1,2)\\
\hline
(1,2) & 3  & (0,0)$\pp$, (0,-1), (-1,0)$^{\oplus 2}$, (1,-1)$\pp$, & (1,1)$\pp$, (1,0),  (0,1)$^{\oplus 2}$, (2,0)$\pp$, \\
	&  &  (2,-1), (-2,1),  (3,-2),  (-3,1)  & (3,0), (-1,2),  (4,-1),  (-2,2) \\	
\hline
	& 5 &  (1,0), (0,0), (-1,0), (1,-1),   &  (2,1), (1,1), (0,1),  (2,0), \\
	&    &  (-1,1), (2,-1), (-2,1) & (0,2), (3,0), (-1,2)\\
\hline
(1,1) & 2  &  (-1,0) &  (0,1)\\
\hline
	& 4 &  (1,0),  (0,0)$\ppp$, (0,-1), (-1,0)$\pp$,  &  (2,1), (1,1)$\ppp$, (1,0), (0,1)$\pp$, 	\\
	&	& (1,-1)$\pp$, (-1,1), (2,-1),  (3,-2),  &	 (2,0)$\pp$, (0,2), (3,0), (4,-1)	\\
	&	& (-2,1)$\pp$, (-3,1), (-3,2) & 	(-1,2)$\pp$, (-2,2), (-2,3)	\\
\hline
(2,0) & 3 & (-1,0), (1,-1)  &  (0,1), (2,0) \\
\hline
	& 5 & (1,0), (0,0)$\ppp$, (0,-1), (-1,0),   & (2,1), (1,1)$\ppp$, (1,0), (0,1),  \\
	&   &  (1,-1),  (-1,1), (2,-1),  (3,-1),	& (2,0),  (0,2), (3,0), (4,0), \\
	&	&   (3,-2), (-2,1),  (-3,1), (-3,2) &  (4,-1), (-1,2), (-2,2), (-2,3)\\
\hline
(0,0) & 0 & (-1,-1)  & (0,0)\\
\hline
	&  2 & (0,0), (0,-1), (-1,0)$\ppp$, (1,-1)$\pp$,  &  (1,1), (1,0), (0,1)$\ppp$, (2,0)$\pp$, \\
	&	& (-1,1), (2,-1), (2,-2), (-2,0) & (0,2), (3,0), (3,-1), (-1,1)\\
	&	& (-2,1), (-3,1), (-4,2) & (-1,2), (-2,2), (-3,3)\\
\hline
	&  4 & (1,0), (0,0)$\ppp$, (-1,1),  (0,-1),     & (2,1), (1,1)$\ppp$, (0,2),  (1,0),   \\
	& 	&  (-1,0), (1,-1), (2,-1), (3,-1),  &  (0,1), (2,0),  (3,0),  (4,0),  \\
	&      & (-2,1), (3,-2),  (-3,1), (-3,2) & (-1,2), (4,-1), (-2,2), (-2,3)\\
\hline
\end{tabular}
\end{center}
\newpage 

\providecommand{\bysame}{\leavevmode\hbox
to3em{\hrulefill}\thinspace}

\end{document}